\documentclass{scrartcl}

\usepackage[english]{babel}
\usepackage{amsmath}
\usepackage{amssymb}
\usepackage{mathtools}
\usepackage{hyperref}
\usepackage{cleveref}
\usepackage{nicefrac}       
\usepackage{hhline}
\usepackage{multirow}

\usepackage[symbol]{footmisc}

\newcommand{\E}{\mathbb{E}}
\DeclarePairedDelimiterX{\expdelim}[1]{[}{]}{#1} 
\newcommand{\Exp}[1]{\E\expdelim*{#1}}
\DeclarePairedDelimiterX{\expconddelim}[2]{[}{]}{#1\,\delimsize\vert\, \mathopen{}#2} 
\newcommand{\Expcond}[2]{\E\expconddelim*{#1}{#2}}

\newcommand{\R}{\mathbb{R}}

\DeclareMathOperator*{\dom}{dom}
\DeclareMathOperator*{\dist}{dist}

\DeclareMathOperator*{\argmax}{arg\,max}
\DeclareMathOperator*{\argmin}{arg\,min}

\newcommand{\prox}[3][]{\operatorname{prox}^{#1}_{#2}\left(#3 \right)}

\usepackage{amsthm}
\theoremstyle{plain}
\newtheorem{theorem}{Theorem}[section]

\newtheorem{lemma}{Lemma}[section]
\newtheorem{proposition}{Proposition}[section]
\newtheorem{algo}{Algorithm}[section]
\newtheorem{assumption}{Assumption}

\theoremstyle{definition}

\theoremstyle{remark}

\crefname{assumption}{Assumption}{Assumptions}
\crefname{equation}{}{} 

\usepackage{xcolor}

\usepackage[utf8]{inputenc}
\usepackage{pifont}
\usepackage{enumerate}

\usepackage{breqn}

\usepackage{graphicx}
\usepackage{caption,subcaption}

\renewcommand{\phi}{\varphi}

\usepackage{mathtools}

\usepackage[normalem]{ulem}

\title{Alternating proximal-gradient steps for (stochastic) nonconvex-concave minimax problems}
\author{
  Radu Ioan Bo{\c t}\footnote{Faculty of Mathematics, University of Vienna, Oskar-Morgenstern-Platz 1, 1090 Vienna, Austria.}
  \and
  Axel B\"ohm$^{*,}$\footnote{Research of Axel B\"ohm supported by the doctoral programme \textit{Vienna Graduate School on Computational Optimization (VGSCO)},
    FWF (Austrian Science Fund), project W 1260.}
  \and \\
  \vspace{-1cm} \\
  \normalsize{\texttt{\{radu.bot, axel.boehm\}@univie.ac.at}}
}

\date{\today}
\begin{document}

\maketitle
\begin{abstract}%
  Minimax problems of the form $\min_x \max_y \Psi(x,y)$ have attracted increased interest largely due to advances in machine learning, in particular generative adversarial networks and adversarial learning. These are typically trained using variants of stochastic gradient descent for the two players.
  Although convex-concave problems are well understood with many efficient solution methods to choose from, theoretical guarantees outside of this setting are sometimes lacking even for the simplest algorithms.
  In particular, this is the case for alternating gradient descent ascent, where the two agents take turns updating their strategies.
  To partially close this gap in the literature we prove a novel global convergence rate for the stochastic version of this method for finding a critical point of $\psi(\cdot) := \max_y \Psi(\cdot,y)$ in a setting which is not convex-concave.
\end{abstract}

\section{Introduction}%
\label{sec:intro}

We investigate the \emph{alternating} variant of gradient descent ascent (GDA) with proximal steps for weakly convex-(strongly) concave saddle point problems, given by 
\begin{equation}
  \label{eq:minimax}
  \min_{x\in\R^d} \max_{y\in\R^n}\, \Big\{\Psi(x,y):=f(x) + \Phi(x,y) - h(y)\Big\}
\end{equation}
for a weakly convex-concave coupling function $\Phi: \R^d\times \R^n \to \R$ and proper, convex and lower semicontinuous (l.s.c.) regularizers $h$ and $f$, see Assumption~\ref{ass:weakly-convex-gradient-lipschitz},~\ref{ass:lipschitz} and~\ref{ass:strongly-concave} for details.

Nonconvex-concave saddle point problems have received a great deal of attention recently due to their application in adversarial learning~\cite{duchi-certifiable}, learning with nondecomposable losses~\cite{fan-top-k-loss,ying-nondecomposable}, learning with uncertain data~\cite{chen-uncertain-data} and generative adversarial imitation learning of linear quadratic regulators~\cite{ho-alt-gda-adversarial-learning}.
Additionally, albeit typically resulting in nonconvex-nonconcave objectives, the large interest in \emph{generative adversarial networks (GANs)}~\cite{GAN,wgan} has led to the studying of saddle point problems under different simplifying assumptions~\cite{bailey2020finite, gidel2019variational, bohm2022two, daskalakis2017training, lin2020gradient}.

In the nonconvex-concave setting \emph{inner loop} methods have received much of the attention~\cite{namkoong2016stochastic, thekumparampil2019efficient, lin2020gradient, kong2021accelerated, ostrovskii2021efficient} with them obtaining the best complexity results in this class, see Table~\ref{tab:rates}. Despite superior theoretical performance these methods have not been as popular in practice, especially in the training of GANs where \emph{single loop} methods are still state-of-the-art~\cite{gidel2019variational, bailey2020finite,daskalakis2017training,GAN,gidel2019negative,ho-alt-gda-adversarial-learning, liu2020decentralized}.
The simplest approach is given by \emph{simultaneous} GDA, which, for a smooth coupling function $\Phi$ and step sizes $\eta_x,\eta_y>0$, reads as:
\begin{equation*}
  \text{(simultaneous)}\quad
  \left\lfloor \begin{array}{l}
           x^+ = x - \eta_x \nabla_x \Phi(x, y) \\
           y^+ = y + \eta_y \nabla_y \Phi(x, y).
         \end{array}\right.
\end{equation*}
After the first step of this method, however, more information is already available, which can be used in the update of the second variable, resulting in
\begin{equation*}
  \text{(alternating)}\quad
  \left\lfloor \begin{array}{l}
           x^+ = x - \eta_x \nabla_x \Phi(x, y) \\
           y^+ = y + \eta_y \nabla_y \Phi(x^+, y).
         \end{array}\right.
\end{equation*}
It has been widely known that the alternating version of GDA has many favorable convergence properties of the simultaneous one~\cite{bailey2020finite,gidel2019negative,xu2020unified,zhang2022near}. It has been long known that for bilinear problems the iterates of simultaneous GDA may diverge while those of the alternating version at least remain bounded. Furthermore,~\cite{gidel2019negative} showed that the alternating version can be made convergent for this simple setting if negative momentum is used, while the same is false for simultaneous GDA.\ In another special setting~\cite{zhang2022near} was able to show better local dependence on the condition number for strongly convex-strongly-concave quadratic problems.
We are naturally interested in --- and will give an affirmative answer to the question:
\begin{center}
  \textbf{Does stochastic alternating GDA have nonasymptotic convergence guarantees for nonconvex minimax problems?}
\end{center}
This might seem surprising as it has been sufficiently demonstrated~\cite{GAN2, gidel2019variational, mesched-GAN-converge, bohm2022two} that both versions of GDA fail to converge for simple bilinear problems if equal step sizes are used. We therefore want to point out the importance of the \emph{two-time-scale} approach which was also emphasized in~\cite{lin2020gradient,heusel-two-time-scale-gan-fid}. However, this alone is also not enough as shown in~\cite{gidel2019negative}. The seeming contradiction is resolved through the observation that our convergence guarantees only concern the $x$-component of the objective function.

\paragraph{Optimality}
For convex-concave minimax problems, the notion of solution is simple. We aim to find a so-called saddle point $(x^*,y^*)\in \R^d\times\R^n$ satisfying
\begin{equation*}
  \label{eq:saddle}
  \Psi(x^*,y) \le \Psi(x^*,y^*) \le \Psi(x,y^*) \quad \forall (x,y) \in \R^d\times\R^n.
\end{equation*}
For convex-concave problems this is equivalent to the first order optimality condition
\begin{equation}
  \label{eq:fop-kkt}
  \left(\begin{array}{c}
          0 \\ 0
  \end{array}\right)
\in{}
\left(\begin{array}{c}
          \nabla_x\Phi(x^*,y^*)\\
          -\nabla_y\Phi(x^*,y^*)
  \end{array}\right)
  +
  \left(\begin{array}{c}
    \partial f(x^*)\\
    \partial h(y^*)
  \end{array}\right).
\end{equation}
Similarly to the nonconvex single objective optimization where one cannot expect to find global minima, if the minimax problem is not convex-concave the notion of saddle point is too strong.
So one natural approach is to focus on conditions such as \cref{eq:fop-kkt}, as done in~\cite{xu2020unified, nouiehed2019solving, hibsa}.
However, treating the two components in such a symmetric fashion might not seem fitting since in contrast to the convex-concave problem $\min_x\max_y \neq \max_y\min_x$.
Instead we will focus, in the spirit of~\cite{lin2020gradient,thekumparampil2019efficient,rafique-provable-minimax}, on the stationarity of what we will refer to as the \emph{max function} given by
\begin{equation}
  \label{eq:max-function}
  \phi(x) := \max_{y\in\R^n}\, \Phi(x,y) - h(y),\quad \text{where $\phi: \R^d \to \R$}.
\end{equation}
This makes sense from the point of view of many practical applications. Problems arising from adversarial learning can be formulated as minimax, but typically only $x$, which corresponds to the classifier is relevant as $y$ is adversarial noise. Similarly, for GANs, one is typically only interested in the generator and not the discriminator.
See Table~\ref{tab:rates} for a comparison of other methods using the same notion of optimality.
Note that it is possible to move from one notion of optimality to the other~\cite{lin2020gradient}, but as both directions are typically associated with additional computational effort a comparison is not trivial and out of scope of this work.

\paragraph{Contributions} We prove novel convergence rates for \emph{alternating} prox-gradient descent ascent for nonconvex-(strongly) concave minimax problems in a deterministic and stochastic setting. 
For deterministic problems,~\cite{xu2020unified} has proved convergence rates for alternating GDA in terms of the criticality of $\Phi$ while we use the max function $\phi$, see \cref{eq:max-function}, instead.
Our results are also more general than e.g.~\cite{lin2020near,zhao2020primal,lin2020gradient} in the sense that they require $\Phi$ to be smooth in the first component wheres we only require weak convexity, similar to~\cite{rafique-provable-minimax}.
Furthermore, we allow for our method to include possibly nonsmooth regularizers, similar to~\cite{zhao2020primal,rafique-provable-minimax}, by passing from a regular projected-gradient to more the more general \emph{proximal-gradient} steps which captures and extends the common constraint setting, necessitating us to prove a more general version of Danskins theorem in the process.

\renewcommand{\arraystretch}{1.2}
\begin{table*}
  \centering
  \caption{
    The gradient complexity of algorithms for nonconvex-(strongly) concave minimax problems for computing  $\epsilon$-stationary points of the max function. $\kappa>0$ is the condition number. The notation $\tilde{\mathcal{O}}$ hides logarithmic terms.} 
  \begin{tabular}{cccccc} 
    \hline 
    & \multicolumn{2}{c}{\textbf{Nonconvex-Strongly Concave}} & \multicolumn{2}{c}{\textbf{Nonconvex-Concave}} & single \\ \cline{2-3} \cline{4-5}
    & deterministic & stocastic & deterministic & stochastic & loop\\
    \hline 
    \cite{rafique-provable-minimax} & $\tilde{\mathcal{O}}(\kappa^2\epsilon^{-2})$ & $\tilde{\mathcal{O}}(\kappa^3\epsilon^{-4})$ & $\tilde{\mathcal{O}}(\epsilon^{-6})$ & $\tilde{\mathcal{O}}(\epsilon^{-6})$ & \ding{55} \\ 
    \cite{zhao2020primal,thekumparampil2019efficient} & -- & -- & $\tilde{\mathcal{O}}(\epsilon^{-3})$ & -- & \ding{55} \\ 
      \cite{lin2020near,ostrovskii2021efficient} & $\tilde{\mathcal{O}}(\sqrt{\kappa}\epsilon^{-2})$ & -- & $\tilde{\mathcal{O}}(\epsilon^{-3})$ & -- & \ding{55} \\ 
    \hline 
    \cite{lin2020gradient} & $\mathcal{O}(\kappa^2\epsilon^{-2})$ & $\mathcal{O}(\kappa^3\epsilon^{-4})$ & $\mathcal{O}(\epsilon^{-6})$ & $\mathcal{O}(\epsilon^{-8})$\footnote[1]{large batchsizes} & \checkmark{} \\ 
    this work & $\mathcal{O}(\kappa^2\epsilon^{-2})$ & $\mathcal{O}(\kappa^3\epsilon^{-4})$ & $\mathcal{O}(\epsilon^{-6})$ & $\mathcal{O}(\epsilon^{-8})$\footnotemark[1] & \checkmark{} \\
    \hline 
  \end{tabular}%
  \label{tab:rates}
\end{table*}

\paragraph{Roadmap} In the remainder of this section we discuss related literature and some real-world applications resulting in nonconvex-concave problems. In Section~\ref{sec:prelim} we discuss the mathematical preliminaries as well as our main assumptions about the involved functions. Section~\ref{sec:non-strongly-concave} and Section~\ref{sec:strongly-concave} are devoted to the setting where the objective function is assumed to be convex and strongly convex, respectively. Both times we treat the deterministic problem first and then the scenario where we are only given a stochastic gradient oracle. Finally, in Section~\ref{sec:experiments} we discussed numerical experiments in adversarial learning. For the interested reader we highlighted the improvements in the analysis of alternating GDA over its simultaneous counterpart in Sections~\ref{sub:altvssim-non-strongly} and~\ref{sub:altvssim-strongly}.

\subsection{Related literature}%

\footnotetext{${}^*$ using large batchsizes of order $\mathcal{O}(\epsilon^{-2})$}
For the purpose of this paper we separate the quantitative study of minimax problems into the following domains.

\paragraph{Convex-concave}
For convex-concave problems historically the \emph{extra-gradient} and the \emph{forward-backward-forward} method have been known to converge. For the former even a rate of $\mathcal{O}(\epsilon^{-1})$ has been proven in~\cite{mirror-prox} under the name of \emph{mirror-prox}. Both of these methods suffer from the drawback of requiring two gradient evaluations per iteration. This has led to the development of methods such as \emph{optimistic GDA}~\cite{daskalakis2017training,daskalakis2018limit} or~\cite{aybat-apd,bohm2022two,gidel2019variational,malitsky2020forward} which use past gradients to reduce the need of gradient evaluations to one per iteration.
In all of these cases, however, convergence guarantees typically do not go beyond the convex-concave setting. Nevertheless, these methods have been employed successfully in the GAN setting~\cite{gidel2019variational, bohm2022two, daskalakis2017training}.

\paragraph{Nonconvex-concave with inner loops}
Approximating the max function by running multiple iterations of a solver on the second component or convexifying the problem by adding a quadratic term and then solving the convex-concave problem constitute natural approaches~\cite{thekumparampil2019efficient, lin2020near, zhao2020primal, rafique-provable-minimax, nouiehed2019solving}. Such methods achieve the best known rates~\cite{thekumparampil2019efficient, lin2020near, zhao2020primal, ostrovskii2021efficient} in this class. However, they are usually quite involved and have for the most part not been used in deep learning applications.

\paragraph{Nonconvex-concave with single loop}
While these methods have received some attention in the training of GANs~\cite{daskalakis2017training,gidel2019variational,bohm2022two} most of the theoretical statement are for convex-concave problems. In the nonconvex setting only two methods have been studied.
Previous research, see~\cite{lin2020gradient,hibsa}, has focused on the \emph{simultaneous} version of the gradient descent ascent algorithm where both components are updated at the same time.
The only other work which focuses on \emph{alternating} GDA is~\cite{xu2020unified}. Their results are in terms of stationarity of $\Phi$ and they do not treat the stochastic case. Note that our work is most similar to~\cite{lin2020gradient} where the same notion of optimality is used and similar rates to our are obtained for \emph{simultaneous} GDA.\

\paragraph{Others}
Clearly the above categories do not cover the entire field. However, other settings have not received as much attention. Only~\cite{xu2020unified} treats (strongly) convex-nonconcave problems and proves convergence rates similar to the nonconvex-(strongly) concave setting. 
In~\cite{tran2020hybrid} a special stochastic nonconvex-linear problem with regularizers is solved via a variance reduced single loop method with a significantly improved rate over the general nonconvex-concave problem.

The most general setting out of all the aforementioned ones is discussed in~\cite{liu2021first, liu2020decentralized, song2020optimistic}, namely the weakly convex-weakly concave setting. They use however, a weaker notion of optimality related to the Minty variational inequality formulation. We also only mentioned (sub)gradient methods, but the restrictive assumption that the proximal operator of a component can be evaluated has been considered as well~\cite{kong2021accelerated}. 

\subsection{Nonconvex-concave applications}%

\subsubsection{Adversarial learning}
Such problems often use an attack model~\cite{madry2017towards} that allows for every pixel to be perturbed up to given threshold $\epsilon$:
\begin{equation*}
  \min_\theta \max_{\Vert z-z_0 \Vert_\infty \le \epsilon}\, \ell(\theta,z),
\end{equation*}
where $z_0$ denotes the ``true'' training examples, and $z$ the adversarial attack. However, this typically leads to nonconvex-nonconcave formulation. So~\cite{namkoong2016stochastic} proposed a distributionally robust model, making use of the Wasserstein distance $W$
\begin{equation*}
   \min_\theta \max_{P :  W(P,P_0)\le \rho}\, \E_P [\ell(\theta,Z)],
\end{equation*}
where $Z\sim P_0 $, which they reformulated via a Lagrangian penalty approach to
\begin{equation}
  \label{eq:duchi-penalty}
  \min_\theta \max_{z}\, \ell(\theta,z) - \gamma \Vert z-z_0 \Vert^2.
\end{equation}
While a larger $\gamma$ corresponds to smaller robustness $\rho$, the model can be made nonconvex-strongly-concave if it is set big enough.

\subsubsection{Generative adversarial imitation learning of linear quadratic regulators}
In imitation learning the objective is to learn from an expert's demonstration of performing a given task. In this case the minimization is performed over the policies with the goal of reducing the discrepancy between the reward of the expert's policy and the proposed one. The maximization is over the parameters of the reward function, see~\cite{ho-alt-gda-adversarial-learning}. If the underlying dynamic and the reward function come from a linear quadratic regulator, see~\cite{imitation-learning-LQR}, this can be expressed as a nonconvex-strongly-concave minimax problem
\begin{equation}
  \min_{K}\max_{\theta}  \, m(K, \theta),
\end{equation}
where $K$ represents the choice of policy and $\theta$ the parameters of the dynamic and reward functions.

\subsubsection{Fair learning}

The work~\cite{fair-learning} observed that a logistic regression model trained on the Fashion-MNIST dataset (comprised of $n=10$ classes) can lead to a bias against certain classes. In order to remove this bias, they proposed to minimize the maximal loss of the different categories, i.e.\
\begin{equation}
  \label{eq:fair}
  \min_{\theta} \max_{1\le i \le n} \, \ell_i(\theta),
\end{equation}
where $\ell_i$ denotes the loss incurred by all examples of class $i$. A similar approach was taken in~\cite{nouiehed2019solving}, but for a more sophisticated CNN model. For the purpose of implementation~\eqref{eq:fair} can be rewritten as
\begin{equation}
  \label{eq:fair-concave}
  \min_{\theta} \max_{(t_1, \dots, t_n) \in \Delta} \, \sum_{i=1}^{n} t_i \ell_i(\theta)
\end{equation}
where $\Delta := \{(t_1, \dots, t_n ): t_i\ge 0, \sum_{i=1}^{n} t_i = 1\}$ denotes the unit simplex. Due to the linearity of~\eqref{eq:fair-concave} in the second variable $(t_1,\dots,t_n)$, the inner maximization problem is in particular concave.

\section{Preliminaries}%
\label{sec:prelim}

As mentioned in the earlier we will consider optimality in terms of the max function for any $x\in\R^d$ given by
 \(\phi(x) := \max_{y\in\R^n}\, \Phi(x,y) - h(y)\). Similarly, we also need the regularized max function
\begin{equation*}
  \psi := \phi+f, \quad \text{where $\psi:\R^d \to \R \cup\{+\infty\}$}.
\end{equation*}
In the remainder of the section we will focus on the necessary preliminaries connected to the weak convexity of the max function in the nonconvex-concave setting, see Section~\ref{sec:non-strongly-concave}.

\subsection{Weak convexity}%
\label{sub:weakly-convex}

In the nonconvex-concave setting of Section~\ref{sec:non-strongly-concave} the max function $\phi$ will in general be nonsmooth, which makes it nonobvious how to define near stationarity.
The max function $\phi$ will, however, turn out to be \emph{weakly convex}, see Proposition~\ref{thm:extended-danskin}.
For some $\rho \ge 0$, we say that
\begin{equation*}
  \label{eq:weakly_convex}
  \mbox{$\psi:\R^d \to \R\cup\{+\infty\}$ is \emph{$\rho$-weakly convex} if $\psi + (\rho/2) \Vert \cdot \Vert^2$ is convex.}
\end{equation*}
An example of a weakly convex function is one which is differentiable and the gradient is uniformly Lipschitz continuous with constant $L$ (we call such a function $L$-smooth). In this case, the weak convexity parameter $\rho$ is given by the Lipschitz constant.

Following~\cite{dima_damek_stoch_weakly_k-4,thekumparampil2019efficient,drusvyatskiy2019efficient,bohm2021variable}, we make use of a smooth approximation of $\psi$ known as the \emph{Moreau envelope} $\psi_\lambda$, parametrized $\lambda>0$.
For a proper, $\rho$-weakly convex and l.s.c.\ function $\psi: \R^d \rightarrow \R\cup\{+\infty\}$, the Moreau envelope of $\psi$ with the parameter $\lambda \in (0,\rho^{-1})$ is the function from $\R^d$ to $\R$ defined by
\begin{equation*}
  \label{eq:moreau}
  \psi_\lambda(x) := \inf_{z \in \R^d} \Big\{ \psi(z) + \frac{1}{2\lambda}\Vert z-x \Vert^2\Big\}.
\end{equation*}
The proximal operator of the function $\lambda \psi$ is the $\arg\min$ of the right-hand side in this definition, that is,
\begin{equation}
  \label{eq:prox}
  \prox{\lambda \psi}{x} := \argmin_{z \in \R^d}\Big\{ \psi(z) + \frac{1}{2\lambda}\Vert z-x \Vert^2\Big\}.
\end{equation}
Note that $\prox{\lambda \psi}{x}$ is uniquely defined by~\eqref{eq:prox} because the function being minimized is proper, l.s.c.\ and strongly convex.
For a function $\psi:\R^d \to \R\cup\{+\infty\}$ and a point $\bar{x}$ such that $\psi(\bar{x})$ is finite, the \emph{Fr\'echet subdifferential} of $\psi$ at $\bar{x}$, denoted by $\partial \psi(\bar{x})$, is the set of all vectors $v \in \R^d$ such that
\begin{equation*}
  \label{eq:subgradient}
  \psi(x) \ge \psi(\bar{x}) + \langle v, x- \bar{x}\rangle + o(\Vert x-\bar{x} \Vert) \quad \text{as $x \to \bar{x}$}.
\end{equation*}
For weakly convex function the Fr\'echet subdifferential can simply be expressed in terms of the convex subdifferential of the (convex) function $\psi+ (\rho/2)\Vert \cdot \Vert^2$.

While the next result is standard for the gradient and convex subgradients we explicitly mention the general case.
\begin{lemma}[{see~\cite[Theorem 3.52]{mordukhovich-variational-analysis}}]%
  \label{lem:lipschitz-bounded-subgradients}
  For an $L_\psi$-Lipschitz continuous function $\psi:\R^d \to \R$ every Fr\'echet subgradient is bounded in norm by $L_\psi$.
\end{lemma}

Now, we provide a useful characterization of the gradient of the Moreau envelope.

\begin{lemma}[{see~\cite[Lemma 2.2]{stoch_weakly_model_based}}]%
  \label{lem:grad_smooth_is_prox}
  Let $\psi: \R^d \to \R\cup\{+\infty\}$ be a proper, $\rho$-weakly convex, and l.s.c.\ function, and let $\lambda \in (0,\rho^{-1})$. Then the Moreau envelope $\psi_\lambda$ is continuously differentiable
  on $\R^d$ with gradient
  \begin{equation*}
    \nabla \psi_\lambda(x) = \frac{1}{\lambda}\left(x - \prox{\lambda \psi}{x}\right) \quad \mbox{for all $x \in \R^d$},
  \end{equation*}
  and this gradient is Lipschitz continuous.
\end{lemma}
In particular, a gradient step with respect to the Moreau envelope corresponds to a proximal step, that is,
\begin{equation}%
  \label{eq:ys9}
  x - \lambda \nabla \psi_\lambda(x) = \prox{\lambda \psi}{x} \quad \mbox{for all $x \in \R^d$}.
\end{equation}


\paragraph{Stationarity}
The Moreau envelope allows us to naturally define a notion of \emph{near stationarity} even for nonsmooth and $\rho$-weakly convex functions.
We say that for an $\epsilon>0$ and a $\lambda\in(0,\rho^{-1})$
\begin{equation}
  \label{eq:epsilon-stationary}
  \mbox{a point $x$ is $\epsilon$-\emph{stationary} for $\psi$ if $\Vert \nabla \psi_\lambda(x) \Vert \le \epsilon$}.
\end{equation}
Canonically, we call a point stationary if the above holds for $\epsilon=0$.
This notion of near stationarity can also be expressed in terms the original function $\psi$.
\begin{lemma}%
  \label{lem:grad-moreau-in-partial-prox}
  Let $x$ be $\epsilon$-stationary for the proper, $\rho$-weakly convex and l.s.c.\ function $\psi$, i.e.\ $\Vert \nabla \psi_\lambda(x) \Vert \le \epsilon$ with $\lambda \in (0,\rho^{-1})$. Then there exist a point $\hat{x}$ such that
  $\Vert x-\hat{x} \Vert\le \epsilon\lambda$ and $\dist(0,\partial \psi(\hat{x}))\le \epsilon$.
\end{lemma}
\begin{proof}
  From the definition of the Moreau envelope, we have that
  \begin{equation*}
    0 \in \partial \psi(\prox{\lambda \psi}{x}) + \frac{1}{\lambda} (\prox{\lambda \psi}{x} - x),
  \end{equation*}
  from which $\nabla \psi_\lambda(x) \in \partial \psi (\prox{\lambda \psi}{x})$
  follows by using~\eqref{eq:ys9}. It is easy to see that $\hat{x}=\prox{\lambda \psi}{x}$ fulfills the required conditions.
\end{proof}

\subsection{About the stochastic setting}%

We discuss the stochastic version of problem~\eqref{eq:minimax} where the coupling function $\Phi$ is actually given as an expectation,
\begin{equation*}
  \Phi(x,y) = \E_{\xi\sim\mathcal{D}}\left[ \Phi(x,y;\xi)\right] \quad \forall (x,y)\in\R^d\times\R^n
\end{equation*}
and we can only access independent samples of the gradient $\nabla_x\Phi(x,y;\xi)$ (or subgradient) and $\nabla_y\Phi(x,y;\zeta)$, where $\xi$ and $\zeta$ are drawn from the (in general unknown) distribution $\mathcal{D}$.

We require the following standard assumption with respect to these sto\-chastic gradient estimators.
\begin{assumption}[unbiased]%
  \label{ass:unbiased}
  The stochastic gradient estimator is unbiased, i.e.
  \begin{equation*}
    \Exp{\nabla\Phi(x,y;\xi)} = \nabla\Phi(x,y) \quad \forall (x,y)\in \R^d\times \R^n,
  \end{equation*}
  or in the case of subgradients
  \begin{equation*}
    \Exp{g^\xi} \in \partial [\Phi(\cdot,y)](x), \quad\text{where $g^\xi \in \partial [\Phi(\cdot,y;\xi)](x)$}.
  \end{equation*}
\end{assumption}

\begin{assumption}[bounded variance]%
  \label{ass:bounded-variance}
  The variance of the estimator is uniformly bounded, i.e.\ for all $(x,y)\in\R^d\times\R^n$ and a variance $\sigma^2\ge0$ we have
  \begin{equation}
    \label{eq:bounded-variance}
    \Exp{\Vert \nabla\Phi(x,y;\xi) - \nabla\Phi(x,y) \Vert^2} \le \sigma^2.
  \end{equation}
  In the setting of Section~\ref{sec:non-strongly-concave} where $\Phi$ is not necessarily smooth in the first component, we make the analogous assumption for subgradients, i.e.\
  \begin{equation}
    \label{eq:bounded-variance-subgradient}
    \Exp{\left\Vert g^\xi-\Exp{g^\xi}  \right\Vert^2} \le \sigma^2,
  \end{equation}
  for a stochastic subgradient $g^\xi\in \partial[\Phi(\cdot,y;\xi)](x)$.
\end{assumption}

\subsection{The algorithm}%

Since we cover different settings such as smooth or not, deterministic and stochastic we try to formulate a unifying scheme.

\begin{algo}[proximal alternating GDA]%
  \label{alg:alternatinggda}
  Let $(x_0,y_0)\in \R^d\times\R^n $ and step sizes $\eta_x,\eta_y > 0$. Consider the following iterative scheme
  \begin{equation*}
    \label{eq:alt-gda-algorithm}
    (\forall k \geq 0) \quad
    \left\lfloor
      \begin{array}{l}
        x_{k+1} = \prox{\eta_x f}{x_k - \eta_x G_x(x_k,y_k)} \\
        y_{k+1} = \prox{\eta_y h}{y_k + \eta_y G_y(x_{k+1},y_k)},
      \end{array}\right.
  \end{equation*}
  where $G_x$ and $G_y$ will be replaced by the appropriate (sub)gradient and its estimator in the deterministic and stochastic setting, respectively. 
\end{algo}

\subsection{Notation}%

We collect different symbols used through this manuscript.

\vspace{0.1cm}
\begin{center}
  \begin{tabular}{| l | l |} 
    \hline 
    \hline 
    Object & Definition \\
    \hline 
    Coupling function & $\Phi(x,y)$ \\
    Objective function & $\Psi(x,y):= f(x) + \Phi(x,y) - h(y)$ \\
    Regularized coupling function & $\Gamma(x,y):=\Phi(x,y) - h(y)$ \\
    Max function & $\phi(x):= \max_{y} \Phi(x,y) - h(y)$ \\
    Regularized max function & $\psi(x):= f(x)+\phi(x)$ \\
    \hline 
    \hline 
  \end{tabular}
\end{center}
\vspace{0.1cm}

\noindent Note that the regularized coupling function $\Gamma$ is only needed in proofs and some technical lemmata. The remaining functions confirm to the logic that small letters denote functions maximized in the second component (and thus only depend on $x$). On the other hand (no matter if capital or not) the letter \emph{psi} indicates the presence of regularizers and \emph{phi} their absence.

\section{Nonconvex-concave objective}%
\label{sec:non-strongly-concave}

In this section we treat the case where the objective function is weakly convex and Lipschitz in $x$, but not necessarily smooth, and concave and smooth in $y$. This will result in a weakly convex and Lipschitz max function whose Moreau envelope we will study for criticality, see~\eqref{eq:epsilon-stationary}.

\subsection{Assumptions}%
\label{sub:ass}

While the first assumption concerns general setting of this section, i.e.\ weakly convex-concave, the latter ones are more of a technical nature.

\begin{assumption}%
  \label{ass:weakly-convex-gradient-lipschitz}
  The coupling function $\Phi$ is
  \begin{enumerate}
    \item[(i)]
      concave and $L_{\nabla\Phi}$-smooth in the second component uniformly in $x$,
      \begin{equation*}
        \Vert \nabla_y\Phi (x,y) - \nabla_y\Phi(x, y') \Vert \le L_{\nabla\Phi} \Vert y - y' \Vert \quad \forall x \in \R^d \, \forall y,y' \in \R^n.
      \end{equation*}
      \item[(ii)]
        $\rho$-weakly convex in the first component uniformly in the second one, i.e.
        \begin{equation*}
          \Phi(\cdot,y) + \frac{\rho}{2}\Vert \cdot \Vert^2 \quad \text{is convex for all $y\in\R^n$}.
        \end{equation*}
  \end{enumerate}
\end{assumption}
Assumption~\ref{ass:weakly-convex-gradient-lipschitz} is fulfilled if e.g.\ $\Phi$ is $L_{\nabla \Phi}$-smooth jointly in both components, i.e.\
\begin{equation*}
    \Vert \nabla \Phi (x,y) - \nabla\Phi(x', y') \Vert \le L_{\nabla\Phi} \Vert (x,y) - (x',y') \Vert \quad \forall x,x' \in \R^d \,\, \forall y,y' \in \R^n,
\end{equation*}
in which case (ii) holds with $\rho=L_{\nabla\Phi}$.

The next assumption is classical in nonconvex optimization.
\begin{assumption}%
  \label{ass:max-function-lower-bounded}
  The function $\psi$ is lower bounded, i.e.\ $\inf_{x\in\R^d}\, \psi(x) > -\infty$.
\end{assumption}
In Section~\ref{sec:non-strongly-concave} we will actually need to bound the Moreau envelope $\psi_\lambda$, but these two conditions are in fact equivalent as for all $x\in\R^d$ and any $\lambda\in (0,\rho^{-1})$
\begin{equation*}
  \psi(x) \ge \psi_\lambda(x) \ge \inf_{u\in\R^d}\, \psi(u).
\end{equation*}
We also want to point out that this is weaker than the lower boundedness of $\Psi$, which is usually required if stationary points of the type~\eqref{eq:fop-kkt} are used, see for example~\cite{hibsa}.

\begin{assumption}%
  \label{ass:lipschitz}
  $\Phi$ is $L$-Lipschitz in the first component uniformly over $\dom h$ in the second one, i.e.
  \begin{equation*}
    \Vert \Phi(x, y) - \Phi(x', y) \Vert \le L \Vert x - x' \Vert \quad \forall x,x' \in \R^d\; \forall y \in \dom h.
  \end{equation*}
\end{assumption}

\begin{assumption}%
  \label{ass:regularizers}
  The regularizers $f$ and $h$ are proper, l.s.c.\ and convex
  \begin{enumerate}
    \item[(i)]
      Additionally, $f$ is either $L_f$-Lipschitz continuous on its domain, which is assumed to be open, or the indicator of a nonempty, convex and closed set.
      Either of those assumptions guarantees for any $\gamma>0$ the bound
      \begin{equation}
        \label{eq:bound-prox-f}
        \Vert \prox{\gamma f}{x}-x \Vert \le \gamma L_f \quad \forall x \in \dom f.
      \end{equation}
    \item[(ii)]
    Furthermore, $h$ has a bounded domain $\dom h$ such that the diameter of $\dom h$ is bounded by $D_h$.
  \end{enumerate}
\end{assumption}

\subsection{Properties of the max function}%
\label{sub:properties-max-function-nonstrongly}

Previous research, when concluding the weak convexity of the max function, has relied on the compactness of the domain over which to maximize. This is done so that the classical Danskin Theorem can be applied. This assumption is e.g.\ fulfilled in the context of \emph{Wasserstein GANs}~\cite{wgan} with weight clipping, but not in other formulations such as~\cite{wgan-gp}.
We provide an extension of the classical Danskin Theorem, which only relies on the concavity and l.s.c.\ of the objective in the second component and the boundedness of $\dom h$, see Assumption~\ref{ass:weakly-convex-gradient-lipschitz} and~\ref{ass:regularizers}.
This implies that for every $x\in\R^d$ the set
\begin{equation}
  \label{eq:solution-set-mapping}
  Y(x) := \Big\{y^* \in \R^n : \phi(x) = \Phi(x,y^*) -h(y^*) = \max_{y\in\R^n} \{\Phi(x,y)-h(y)\} \Big\}
\end{equation}
is nonempty. For brevity we denote arbitrary elements of $Y(x_k)$ by $y_k^*$ for all $k\ge0$.

\begin{proposition}[Subgradient characterization of the max function]%
  \label{thm:extended-danskin}
  Let Assumption~\ref{ass:weakly-convex-gradient-lipschitz} and~\ref{ass:regularizers} hold true.
  Then, the function $\phi$, see~\eqref{eq:max-function}, fulfills
  \begin{equation*}
    \partial [\Phi(\cdot,y^*)](x) \subseteq \partial\phi(x) \quad \forall y^*\in Y(x), \forall x\in \R^d.
  \end{equation*}
  In particular, $\phi$ is $\rho$-weakly convex.
\end{proposition}
\begin{proof}
  From the $\rho$-weak convexity of $\Phi(\cdot,y)$, we have that $\Phi(\cdot,y) + \frac{\rho}{2} \|\cdot\|^2$ is convex for all $y\in\R^n$. We define $\tilde{\Phi}(x,y) = \Phi(x,y) + \frac\rho2 \|x\|^2$ and $\tilde{\Gamma}(x,y) = \Gamma(x,y)+\frac{\rho}{2}\Vert x \Vert^2$ for $(x,y)\in\R^d\times\R^n$ as well as
  \begin{equation*}
    \tilde{\phi}(x) = \max_{y \in \R^n} \tilde{\Gamma}(x,y) = \phi(x) + \frac{\rho}{2} \|x\|^2.
  \end{equation*}
  Notice that $\tilde{\Gamma}(x,\cdot)$ is concave for any $x \in \R^d$ and $\tilde{\Gamma}(\cdot,y)$ is convex for any $y \in \R^n$.
  Thus, the function $\tilde{\phi}$ is convex and $\dom \tilde{\phi} = \dom \phi = \R^d$. Therefore $\phi$ is continuous, which implies that $\partial \phi(x) \neq \emptyset$ for any $x \in \R^d$.
  Let $x \in \R^d$, $y \in Y(x)$ and $v \in \R^d$. For any $\alpha >0$ it holds
  \begin{equation*}
    \frac{\tilde{\phi}(x+\alpha v) - \tilde{\phi}(x)}{\alpha} \geq \frac{\tilde{\Gamma}(x+\alpha v, y) - \tilde{\Gamma}(x,y)}{\alpha} =  \frac{\tilde{\Phi}(x+\alpha v, y) - \tilde{\Phi}(x,y)}{\alpha},
  \end{equation*}
  thus
  \begin{equation}
    \begin{aligned}
      \tilde{\phi}'(x;v) &= \inf_{\alpha >0} \frac{\tilde{\phi}(x+\alpha v) - \tilde{\phi}(x)}{\alpha} \geq \inf_{\alpha >0} \frac{\tilde{\Phi}(x+\alpha v, y) - \tilde{\Phi}(x,y)}{\alpha} = [\tilde{\Phi}(\cdot,y)]'(x;v),
    \end{aligned}
  \end{equation}
  where $[\Phi(\cdot,y)]'(x;v)$ denotes the directional derivative of $\Phi$ in the first component at $x$ in the direction $v$.
  In conclusion,
  \begin{equation}
    \label{eq1}
    \tilde{\phi}'(x;v) \geq \sup_{y \in Y(x)}  [\tilde{\Phi}(\cdot,y)]'(x;v) \quad \forall v \in \R^d
  \end{equation}
  and for $y\in Y(x)$ we therefore conclude $\partial [\tilde{\Phi}(\cdot,y)](x) \subseteq \partial\tilde{\phi}(x)$.
  The first statement is obtained by subtracting $\rho x$ on both sides of the inclusion.
\end{proof}

\begin{lemma}[Lipschitz continuity of the max function]%
  \label{lem:max-is-lipschitz}
  The Lipschitz continuity of $\Phi$ in its first component implies that $\phi$ is Lipschitz with the same constant.
\end{lemma}
\begin{proof}
  Let $x,x' \in \R^d$ and $y^*\in Y(x)$. On the one hand
  \begin{equation*}
    \begin{aligned}
      \phi(x) - \phi(x') &= \Phi(x,y^*)-h(y^*) - \phi(x')\\
                        &\le \Phi(x,y^*)-h(y^*) - \Phi(x',y^*)+h(y^*)\le L \Vert x-x' \Vert.
    \end{aligned}
  \end{equation*}
  The reverse direction $\phi(x') - \phi(x) \le L\Vert x-x' \Vert$ follows analogously.
\end{proof}

\subsection{Deterministic setting}%
\label{sub:deterministic}

For initial values $(x_0,y_0)\in \dom f \times \dom h$ the deterministic version of alternating GDA, for $g_k \in \partial [\Phi(\cdot,y_k)](x_k)$, reads as
\begin{equation}%
  \label{eq:alt-gda-non-strongly}
  (\forall k \ge 0)
  \left\lfloor
    \begin{array}{l}
      x_{k+1} = \prox{\eta_x f}{x_k - \eta_x g_k} \\
      y_{k+1} = \prox{\eta_y h}{y_k + \eta_y \nabla_y\Phi(x_{k+1},y_k)}.
    \end{array}\right.
\end{equation}

\begin{theorem}%
  \label{thm:alternating-rate}
  Let Assumption~\ref{ass:weakly-convex-gradient-lipschitz},~\ref{ass:max-function-lower-bounded},~\ref{ass:lipschitz} and~\ref{ass:regularizers} hold true.
  For algorithm~\eqref{eq:alt-gda-non-strongly} with the step sizes
  \begin{equation*}
  \eta_x=\min \left\{\frac{\epsilon^4}{L_{\nabla\Phi}\rho^2 D_h^2{(L+L_f)}^2}, \frac{\epsilon^2}{\rho L^2}\right\}, \eta_y=\frac{1}{L_{\nabla\Phi}} \,\, \text{and}\,\, \lambda = \frac{1}{2\rho}
  \end{equation*}
  the number of gradient evaluations $K$ required is
  \begin{equation*}
    \begin{aligned}
      \mathcal{O}\left( \frac{\Delta^* \rho {(L+L_f)}^2}{\epsilon^4} \max \Big\{1, \frac{\rho L_{\nabla\Phi} D_h^2}{\epsilon^2} \Big\} + \frac{\rho \Delta_0}{\epsilon^2}\right),
    \end{aligned}
  \end{equation*}
  to visit an $\epsilon$-stationary given by $\min_{1\le k\le K}\Vert \nabla \psi_\lambda(x_k) \Vert \le \epsilon$, where $\Delta^* := \psi(x_0) - \inf_{x\in\R^d} \psi(x)$ and $\Delta_0 :=\psi(x_0)-\Psi(x_0,y_0)$.
\end{theorem}

Similarly to the proofs in~\cite{dima_damek_stoch_weakly_k-4,lin2020gradient} and others, the main descent statement makes use of the quantity $\prox{\lambda \psi}{x_k}$ for a $\lambda>0$. This is somewhat surprising as this point does not appear in the algorithm and can in general not be computed.

But first, we need to establish the fact that $\hat{x}_k := \prox{\lambda \psi}{x_k}$ can also be written as the proximal operator of $f$ evaluated at an auxiliary point.
\begin{lemma}%
  \label{lem:xhatk-as-prox}
  For any $\lambda\in(0,\rho^{-1})$ and all $k\ge0$ the point $\hat{x}_k := \prox{\lambda g}{x_k}$ can also be written for some $v_k\in \partial \phi(\hat{x}_k)$ as
  \begin{equation*}
    \hat{x}_k = \prox{\eta_x f}{\eta_x\lambda^{-1}x_k - \eta_x v_k + (1-\eta_x\lambda^{-1})\hat{x}_k}.
  \end{equation*}
\end{lemma}
\begin{proof}
  Let $k\ge0$ be arbitrary but fixed and recall that $g=f+\phi$. By the definition of $\hat{x}_k$ we have that
  \begin{equation*}
      0 \in \partial \psi(\hat{x}_k) + \frac1\lambda(\hat{x}_k-x_k) = \partial(\phi+f)(\hat{x}_k) + \frac1\lambda(\hat{x}_k-x_k).
  \end{equation*}
  We can estimate through the continuity of $\phi$ and subdifferential calculus
  \begin{equation*}
    \frac1\lambda(x_k - \hat{x}_k) \in \partial(\phi+f)(\hat{x}_k) \subseteq \partial \phi(\hat{x}_k) + \partial f(\hat{x}_k).
  \end{equation*}
  Thus, there exists $v_k \in \partial \phi(\hat{x}_k)$ such that $\frac1\lambda (x_k - \hat{x}_k) \in v_k + \partial f(\hat{x}_k)$.
  Also,
  \begin{equation*}
    \begin{aligned}
      \frac1\lambda(x_k - \hat{x}_k) \in  \partial f(\hat{x}_k) +  v_k &\Leftrightarrow \frac{\eta_x }{\lambda}x_k - \eta_x v_k + (1-\frac{\eta_x}{\lambda})\hat{x}_k \in \hat{x}_k + \eta_x \partial f(\hat{x}_k) \\
      &\Leftrightarrow \hat{x}_k = \prox{\eta_x f}{\frac{\eta_x}{\lambda}x_k - \eta_x v_k + \big(1-\frac{\eta_x}{\lambda}\big)\hat{x}_k}.
    \end{aligned}
  \end{equation*}
\end{proof}

With the previous lemma in place we can now turn our attention to the first step of the actual convergence proof.

\begin{lemma}%
  \label{lem:descent-on-moreau}
  With $\lambda = \nicefrac{1}{2\rho}$ and $\eta_x \ge 0$ we have for all $k\ge0$ that
  \begin{equation*}
    \psi_\lambda(x_{k+1}) \le \psi_\lambda(x_k) + 2\rho\eta_x\Delta_k - \frac12\eta_x \Vert \nabla \psi_\lambda(x_k) \Vert^2  + 4\rho\eta_x^2L^2,
  \end{equation*}
  where $\Delta_k := \psi(x_k) - \Psi(x_k, y_k) \ge 0$.
\end{lemma}
\begin{proof}
  Let $k\ge0$ be fixed. As before we denote $\hat{x}_k = \prox{\lambda \psi}{x_k}$. From the definition of the Moreau envelope we have that
  \begin{equation}
    \label{eq:moreau-is-min}
    \psi_\lambda(x_{k+1}) = \min_{x\in\R^d} \Big\{ \psi(x) + \frac{1}{2\lambda} \Vert x - x_{k+1} \Vert^2\Big\} \le \psi(\hat{x}_k) + \frac{1}{2\lambda} \Vert \hat{x}_k - x_{k+1} \Vert^2.
  \end{equation}
  Let now $v_k \in \partial\phi(\hat{x}_k)$ as in Lemma~\ref{lem:xhatk-as-prox}. We successively deduce for $\beta:=1-\eta_x \lambda^{-1}$
  \begin{align}
    \Vert \hat{x}_k - x_{k+1} \Vert^2 &= \Vert \prox{\eta_x f}{\eta_x\lambda^{-1}x_k - \eta_x v_k + \beta\hat{x}_k} - \prox{\eta_x f}{x_k - \eta_x g_k}\Vert^2 \label{eq:defin-iter}\\
     &\le \Vert \beta(\hat{x}_k - x_k) + \eta_x (g_k - v_k)\Vert^2 \label{eq:1-Lipschitz}\\
     &= {\beta}^2\Vert \hat{x}_k - x_k \Vert^2 + 2 \eta_x\beta\langle g_k-v_k, \hat{x}_k - x_k\rangle
     + \eta_x^2\Vert g_k - v_k \Vert^2 \notag\\
     &\le {\beta}^2\Vert \hat{x}_k - x_k \Vert^2 + 2 \eta_x \beta\langle g_k-v_k, \hat{x}_k - x_k\rangle + 4\eta_x^2L^2 \label{eq:Lipschitz}
  \end{align}
  where~\eqref{eq:defin-iter} uses Lemma~\ref{lem:xhatk-as-prox} and the definition of $x_{k+1}$, inequality~\eqref{eq:1-Lipschitz} holds because of the nonexpansiveness of the proximal operator, and~\eqref{eq:Lipschitz} follows from the Lipschitz continuity of $\Phi$ and $\phi$ (see Lemma~\ref{lem:max-is-lipschitz}) and the fact that Lipschitz continuity implies bounded subgradients.
  We are left with estimating the inner product in the above inequality and we do so by splitting it into two:
  first of all, from the weak convexity of $\Phi$ in $x$ we have that
  \begin{equation*}
    \begin{aligned}
      \langle g_k, \hat{x}_k - x_k \rangle &\le \Phi(\hat{x}_k, y_k) - \Phi(x_k, y_k) + \frac{\rho}{2} \Vert \hat{x}_k - x_k \Vert^2 \\
      &\le \phi(\hat{x}_k) - \Gamma(x_k,y_k) + \frac{\rho}{2} \Vert \hat{x}_k - x_k \Vert^2.
    \end{aligned}
  \end{equation*}
  Secondly, by the $\rho$-weak convexity of $\phi$
  \begin{equation*}
    -\langle v_k, \hat{x}_k - x_k\rangle \le \phi(x_k) - \phi(\hat{x}_k) + \frac{\rho}{2} \Vert \hat{x}_k - x_k \Vert^2.
  \end{equation*}
  Combining the last two inequalities we get that
  \begin{equation}
    \label{eq:estimate-middle}
    \langle g_k-v_k, \hat{x}_k - x_k\rangle \le \phi(x_k) - \Gamma(x_k,y_k) + \rho \Vert \hat{x}_k - x_k \Vert^2.
  \end{equation}
  Plugging~\eqref{eq:estimate-middle} into~\eqref{eq:Lipschitz} we deduce
  \begin{equation}
    \begin{aligned}
    \label{eq:ab3}
    \Vert \hat{x}_k - x_{k+1} \Vert^2 \le \underbrace{[{(1-\eta_x\lambda^{-1})}^2+2\eta_x(1-\eta_x \lambda^{-1})\rho]}_{= (*)} \Vert \hat{x}_k - x_k \Vert^2 + 2\eta_x \Delta_k + 4\eta_x^2L^2, \\
    \end{aligned}
  \end{equation}
  where we used the fact that $1-\beta\le1$ in the factor of $\Delta_k$.
  Now note that
  \begin{equation}
    \label{eq:terrible-algebra}
    \begin{aligned}
      (*) &= 1 - 2\eta_x \lambda^{-1} + \eta_x^2 \lambda^{-2} + 2\eta_x\rho - 2\eta_x^2 \lambda^{-1}\rho\\
      &= 1 - 4\eta_x\rho + 4 \eta_x^2\rho^2 + 2\eta_x\rho - 4 \eta_x^2\rho^2 = 1-2\eta_x\rho.
    \end{aligned}
  \end{equation}
  Combining~\eqref{eq:moreau-is-min},~\eqref{eq:ab3} and~\eqref{eq:terrible-algebra} we deduce, using $\lambda=\nicefrac{1}{2\rho}$,
  \begin{equation*}
    \label{eq:descent-moreau}
    \begin{aligned}
      \psi_\lambda(x_{k+1}) &\le \psi(\hat{x}_k) + \frac{1}{2\lambda} \left(\Vert \hat{x}_k - x_k \Vert^2 + 2\eta_x\Delta_k -2\eta_x\rho \Vert\hat{x}_k - x_k \Vert^2  + 4\eta_x^2L^2\right) \\
      &= \psi_\lambda(x_k) + 2\rho\eta_x\Delta_k - \frac12\eta_x \Vert \nabla \psi_\lambda(x_k) \Vert^2  + 4\rho\eta_x^2L^2.
    \end{aligned}
  \end{equation*}
\end{proof}

Naturally, we want to telescope the inequality established by the previous lemma. We are left with estimating $\Delta_k$, preferably even in a summable way. But first we need the following technical, yet standard lemma, estimating the amount of increase obtained by a single iteration of gradient ascent.
\begin{lemma}%
  \label{lem:standard-gradient-ascent-estimation}
  It holds for all $y\in\R^n$ and $k\ge0$ that
  \begin{equation}
    \label{eq:gradient-ascent-estimation}
    \Psi(x_{k+1}, y) - \Psi(x_{k+1}, y_{k+1}) \le  \frac{1}{2\eta_y} \Big(\Vert y - y_k \Vert^2 - \Vert y - y_{k+1} \Vert^2 \Big).
  \end{equation}
\end{lemma}
\begin{proof}
  This is a standard estimate on the improvement made by a single prox-gradient step for a convex (in this case concave) function, see for example~\cite[Lemma 2.3]{fista}.
\end{proof}

We can now use the previous lemma to estimate $\Delta_k$. Recall also that $y^*_k$ denotes a maximizer of $\Psi(x_k,\cdot)$ for all $k\ge0$.
\begin{lemma}%
  \label{lem:estimate-Delta}
  We have that for all $1 \le m \le k$,
  \begin{equation}
    \begin{aligned}
      \label{eq:estimate-Delta}
      \Delta_k \le 2\eta_x L(L+L_f)(k-m) + \frac{1}{2\eta_y} \Big(\Vert y_{k-1} - y^*_m \Vert^2 - \Vert y_k - y^*_m \Vert^2 \Big).
    \end{aligned}
  \end{equation}
\end{lemma}
\begin{proof}
  Plugging $y=y^*_m$ into~\eqref{eq:gradient-ascent-estimation} we deduce that
  \begin{equation}
    \label{eq:standard_descent_with_ystar}
    \begin{aligned}
       0 \le \Psi(x_k, y_k) - \Psi(x_k, y^*_m) + \frac{1}{2\eta_y} \Big(\Vert y^*_m - y_{k-1} \Vert^2 - \Vert y^*_m - y_k \Vert^2 \Big).
    \end{aligned}
  \end{equation}
  Starting from the definition of $\Delta_k = \Psi(x_k,y^*_k) - \Psi(x_k,y_k)$, we add~\eqref{eq:standard_descent_with_ystar} to obtain
  \begin{equation}
    \label{eq:diff-function-values}
    \begin{aligned}
      \Delta_k 
      &\le \Psi(x_k,y^*_k) - \Psi(x_k,y_m^*) + \frac{1}{2\eta_y} \Big(\Vert y^*_m - y_{k-1} \Vert^2 - \Vert y^*_m - y_k \Vert^2 \Big).
    \end{aligned}
  \end{equation}
  Due to the Lipschitz continuity of $\Phi$, terms which only differ in their first argument will be easy to estimate. Therefore, we insert and subtract $\Phi(x_m,y_{k}^*)$ to deduce
  \begin{equation}
    \label{eq:plus-minus-functionval}
    \begin{aligned}
      \MoveEqLeft \Psi(x_k, y^*_k) - \Psi(x_k, y^*_m)  \\ 
      &= \Phi(x_k, y^*_k) - \Phi(x_m, y^*_k) + \Phi(x_m, y^*_k)-h(y^*_k) - \Phi(x_k, y^*_m) + h(y^*_m) \\
      &\le \Phi(x_k, y^*_k) - \Phi(x_m, y^*_k) + \Phi(x_m, y^*_m)-h(y^*_m) - \Phi(x_k, y^*_m) + h(y^*_m)\\
      &= \Phi(x_k,y^*_k) - \Phi(x_m,y^*_k) + \Phi(x_m,y^*_m) - \Phi(x_k, y^*_m).
    \end{aligned}
  \end{equation}
  We estimate the above expression for $k>m$ by making use of the Lipschitz continuity of $\Phi(\cdot, y)$ and~\eqref{eq:bound-prox-f} deducing
  \begin{equation}
    \label{eq:fun-val-same-y}
    \begin{aligned}
      \MoveEqLeft \Phi(x_k, y^*_k) - \Phi(x_m, y^*_k)  \le L \Vert x_k - x_m \Vert \le L \sum_{l=m}^{k-1} \Vert x_{l+1}-x_l \Vert \\
      &\le L \sum_{l=m}^{k-1} \Big( \Vert \prox{\eta_x f}{x_l -\eta_x g_l} - \prox{\eta_x f}{x_l} \Vert + \Vert \prox{\eta_x f}{x_l} - x_l \Vert \Big)\\
      & \le \eta_x L(L+L_f)(k-m).
    \end{aligned}
  \end{equation}
  For $k=m$ the inequality follows trivially. Analogously, we deduce
  \begin{equation}
    \label{eq:fun-val-same-y-2}
      \Phi(x_m, y^*_m) - \Phi(x_k, y^*_m) \le \eta_x L(L+L_f)(k-m).
  \end{equation}
  Plugging~\eqref{eq:plus-minus-functionval},~\eqref{eq:fun-val-same-y} and~\eqref{eq:fun-val-same-y-2} into~\eqref{eq:diff-function-values} gives the statement of the lemma.
\end{proof}

In order to estimate the summation of $\Delta_k$ we will use a trick to sum over it in blocks, where the size $B$ of these blocks will depend on the total number of iterations $K$. Note that w.l.o.g.\ we assume that the block size $B \le K$ divides $K$ without remainder.

\begin{lemma}%
  \label{lem:estimate-sum-of-delta}
  It holds that for all $K\ge1$
  \begin{equation}
    \label{eq:estimate-sum-of-delta}
    \frac{1}{K} \sum_{k=0}^{K-1} \Delta_k \le \eta_x L(L+L_f) B + \frac{L_{\nabla\Phi}D_h^2}{2B} + \frac{\Delta_0}{K}.
  \end{equation}
\end{lemma}
\begin{proof}
  By splitting the summation into blocks we get that
  \begin{equation}
    \label{eq:split--summation-into-blocks}
    \sum_{k=0}^{K-1} \Delta_k  = \sum_{j=0}^{K/B-1} \ \sum_{k=jB}^{(j+1)B-1} \Delta_k.
  \end{equation}
  By using~\eqref{eq:estimate-Delta} from Lemma~\ref{lem:estimate-Delta} with $j>0$ and $m=jB$ and the fact that $\sum_{k=1}^{B-1} k \le \nicefrac{B^2}{2}$ we have
  \begin{equation}
    \label{eq:sum-over-any-block}
    \begin{aligned}
      \sum_{k=jB}^{(j+1)B-1} \Delta_k &\le \eta_x L(L+L_f) B^2 + \frac{1}{2\eta_y} \Vert y_{jB-1}-y^*_{jB} \Vert^2 \\
    \end{aligned}
  \end{equation}
  where the last term can be bounded by $D_h^2$, which was defined in Assumption~\ref{ass:regularizers} and denotes the diameter of $\dom h$.
  We do the same for the case $j=0$ but choose here $m=1$ and have separate out the first summand of $\sum_{k=0}^{B-1} \Delta_k$ as Lemma~\ref{lem:estimate-Delta} does not hold for $k=0$ and therefore get an extra $\Delta_0$ summand.
  where $D_h$ was defined in Assumption~\ref{ass:regularizers} and denotes the diameter of $\dom h$.
  Plugging~\eqref{eq:sum-over-any-block} into~\eqref{eq:split--summation-into-blocks} gives
  \begin{equation*}
    \frac{1}{K} \sum_{k=0}^{K-1} \Delta_k \le \eta_x L(L+L_f) B + \frac{1}{2\eta_y B}D_h^2 + \frac{\Delta_0}{K}.
  \end{equation*}
  The desired statement is obtained by using the step size $\eta_y=\nicefrac{1}{L_{\nabla\Phi}}$.
\end{proof}

\begin{proof}[Proof of Theorem~\ref{thm:alternating-rate}]
  From Lemma~\ref{lem:descent-on-moreau} we deduce by summing up
  \begin{equation*}
    \psi_\lambda(x_{K}) \le \psi_\lambda(x_0) + 2\eta_x\rho\sum_{k=0}^{K-1} \Delta_k - \frac{1}{2}\eta_x\sum_{k=0}^{K-1} \Vert\nabla \psi_\lambda(x_k)\Vert^2 + 4K\rho\eta_x^2L^2.
  \end{equation*}
  Next, we divide by $K$ and obtain that
  \begin{equation*}
    \frac{1}{K}\sum_{k=0}^{K-1}\Vert \nabla \psi_\lambda(x_k) \Vert^2  \le 2\frac{\Delta^*}{\eta_x K} + \frac{4\rho}{K}\sum_{k=0}^{K-1} \Delta_k + 8\rho\eta_x L^2.
  \end{equation*}
  Now, we plug in~\eqref{eq:estimate-sum-of-delta} to deduce that
  \begin{equation*}
    \frac{1}{K}\sum_{k=0}^{K-1}\Vert \nabla \psi_\lambda(x_k) \Vert^2  \le 2\frac{\Delta^*}{\eta_x K} + 4\rho\Big( \eta_x L(L+L_f) B + \frac{L_{\nabla\Phi}D_h^2}{2B} \Big) + \frac{4\rho\Delta_0}{K} + 8\eta_x\rho L^2.
  \end{equation*}
  With $B= \frac{D_h}{L}\sqrt{\frac{L_{\nabla\Phi}}{\eta_x}}$, we have that
  \begin{equation*}
    \frac{1}{K}\sum_{k=0}^{K-1}\Vert \nabla \psi_\lambda(x_k) \Vert^2 \le 2\frac{\Delta^*}{\eta_x K} + 6\rho\sqrt{L_{\nabla\Phi}\eta_x}D_h(L+L_f) + \frac{4\rho\Delta_0}{K} + 8\eta_x\rho L^2.
  \end{equation*}
  By plugging in the step size described in the statement of the theorem we obtain
  \begin{equation*}
    \frac{1}{K}\sum_{k=0}^{K-1}\Vert \nabla \psi_\lambda(x_k) \Vert^2 \le 2\frac{\Delta^*}{K}\max\left\{\frac{L^2 \rho}{\epsilon^2}, \frac{L_{\nabla\Phi}\rho^2 D_h^2 {(L+L_f)}^2}{\epsilon^4}\right\} + \frac{4\rho\Delta_0}{K} + 14 \epsilon^2.
  \end{equation*}
  proving the desired complexity result.
\end{proof}

\subsection{Stochastic setting}%
\label{sub:stochastic}

For initial values $(x_0,y_0)\in \dom f \times \dom h$ the stochastic version of alternating GDA is given by
\begin{equation}%
  \label{eq:gda-non-strongly-stoch}
  (\forall k\ge 0)
  \left\lfloor
    \begin{array}{l}
      x_{k+1} = \prox{\eta_x f}{x_k - \eta_x g_k^\xi} \\
      y_{k+1} = \prox{\eta_y h}{y_k + \eta_y \nabla_y\Phi(x_{k+1},y_k;\zeta_k)},
    \end{array}\right.
\end{equation}
for $g_k^\xi \in \partial[\Phi(\cdot,y_k;\xi_k)](x_k)$ for $\xi_k,\zeta_k\sim\mathcal{D}$ independent from all previous iterates.\\

\begin{theorem}%
  \label{thm:alternating-rate-stoch}
  Let in addition to the assumptions of Theorem~\ref{thm:alternating-rate} also Assumption~\ref{ass:unbiased} and~\ref{ass:bounded-variance} hold true. For algorithm~\eqref{eq:gda-non-strongly-stoch} with step sizes
  \begin{equation*}
  \eta_x = \min \left\{\frac{\epsilon^2}{\rho(L^2+\sigma^2)}, \frac{\epsilon^4}{\rho^2 L(L+L_f+\sigma)D_h^2 L_{\nabla\Phi}}, \frac{\epsilon^6}{\rho^3L(L+L_f+\sigma)\sigma^2D_h^2 }\right\},
  \end{equation*}
  $\eta_y= \min\{\frac{1}{2 L_{\nabla\Phi}}, \frac{\epsilon^2}{\rho \sigma^2}\}$ and $\lambda=\frac{1}{2\rho}$ the number of stochastic gradient evaluations $K$ required is
  \begin{equation*}
      \mathcal{O}\left(\frac{\Delta^*(L^2+L_f^2+\sigma^2)\rho}{\epsilon^4}\max\Big\{1, \frac{\rho L_{\nabla\Phi}LD_h^2}{\epsilon^2}, \frac{\rho^2 D_h^2 \sigma^2}{\epsilon^4}\Big\}+ \frac{\Delta_0\rho}{\epsilon^2}\right),
  \end{equation*}
  where $\Delta^* = \psi(x_0) - \inf_{x\in\R^d} \psi(x)$ and $\Delta_0=\psi(x_0)-\Psi(x_0,y_0)$, to visit an $\epsilon$-stationary point in expectation such that $\min_{1\le k \le K}\E[\Vert \nabla \psi_\lambda(x_k) \Vert] \le \epsilon$.
\end{theorem}

The proof proceeds along the same lines of the deterministic case. Similarly we show an adapted version of Lemma~\ref{lem:descent-on-moreau}.

\begin{lemma}%
  \label{lem:stoch-descent-on-moreau}
  With $\lambda = \nicefrac{1}{2\rho}$ we have for all $k\ge0$ that
  \begin{equation*}
    \Exp{\psi_\lambda(x_{k+1})} \le \Exp{\psi_\lambda(x_k)} + 2\rho\eta_x\hat{\Delta}_k - \frac{\eta_x}{2}\Exp{\Vert \nabla \psi_\lambda(x_k) \Vert^2} + 4\rho\eta_x^2(L^2+\sigma^2)
  \end{equation*}
  where $\hat{\Delta}_k := \Exp{\psi(x_k) - \Psi(x_k, y_k)}$.
\end{lemma}
\begin{proof}
  Let $k\ge0$ be arbitrary but fixed. It follows easily from~\eqref{eq:bounded-variance-subgradient} that
  \begin{equation}
    \label{eq:steiner}
    \Exp{\Vert g_k^\xi \Vert^2} \le \Exp{\Vert g_k \Vert^2} + \sigma^2 \le L^2 + \sigma^2,
  \end{equation}
  where $\Exp{g_k^\xi}=g_k \in \partial_x\Phi(x_k,y_k)$.
  The definition of the Moreau envelope yields
  \begin{equation}
    \label{eq:defin-moreau}
    \Exp{g_\lambda(x_{k+1})} \le \Exp{\psi(\hat{x}_k)} + \frac{1}{2\lambda}\Exp{\Vert \hat{x}_k - x_{k+1} \Vert^2}.
  \end{equation}
  Similarly to Lemma~\ref{lem:descent-on-moreau} we deduce that for $v_k\in \partial\phi(\hat{x}_k)$ (as given in Lemma~\ref{lem:xhatk-as-prox}) and $\beta=1-\eta_x \lambda^{-1}$
  \begin{equation*}
    \begin{aligned}
      \Vert\hat{x}_k - x_{k+1} \Vert^2 
       &= {\beta}^2\Vert \hat{x}_k - x_k \Vert^2 + 2 \eta_x \beta\langle g_k^\xi - v_k, \hat{x}_k - x_k\rangle
       + \eta_x^2\Vert g_k^\xi - v_k \Vert^2.
    \end{aligned}
  \end{equation*}
  By applying the conditional expectation $\Expcond{\cdot}{x_k,y_k}$, then the unconditional one and using~\eqref{eq:steiner}, we get that
  \begin{equation*}
    \Exp{\Vert \hat{x}_k - x_{k+1} \Vert^2} \le \Exp{\Vert \hat{x}_k - x_k \Vert^2 + 2 \eta_x(1-\frac{\eta_x}{\lambda}) \langle g_k-v_k, \hat{x}_k - x_k\rangle} + 4\eta_x^2(L^2+\sigma^2).
  \end{equation*}
  where $g_k= \Exp{g_k^\xi}$.
  Lastly, we combine the above inequality with~\eqref{eq:defin-moreau} and the estimate for the inner product~\eqref{eq:estimate-middle} as in Lemma~\ref{lem:descent-on-moreau} to deduce the statement of the lemma.
\end{proof}

Next, we discuss the stochastic version of Lemma~\ref{lem:standard-gradient-ascent-estimation}. It is clear that we cannot expect the same amount of function value increase by a single iteration of gradient ascent if we do not use the exact gradient.

\begin{lemma}%
  \label{lem:stoch-standard-gradient-ascent-estimation}
  With $\eta_y\le \nicefrac{1}{2L_{\nabla\Phi}}$ we have for all $k\ge0$ and all $y \in \R^n$
  \begin{equation}
    \label{eq:stoch-gradient-ascent-estimation}
    \begin{aligned}
      \MoveEqLeft \Exp{\Psi(x_{k+1}, y) - \Psi(x_{k+1}, y_{k+1})} \le  \frac{1}{2\eta_y} \Big(\Exp{\Vert y^*_m - y_k \Vert^2 - \Vert y^*_m - y_{k+1} \Vert^2} \Big) + \eta_y\sigma^2.
    \end{aligned}
  \end{equation}
\end{lemma}
\begin{proof}
  Let $k\ge0$ and $y\in\R^n$ be arbitrary but fixed.
  As in Lemma~\ref{lem:standard-gradient-ascent-estimation}, we deduce
  \begin{equation*}
    \begin{aligned}
      h(y_{k+1}) - \langle \nabla_y\Phi(x_{k+1},y_k;\zeta_k), y_{k+1}-y_k \rangle + \frac{1}{2\eta_y} \Big(\Vert y_{k+1} -y_k \Vert^2 + \Vert y-y_{k+1} \Vert^2 \Big)  \\
      \le h(y) - \langle \nabla_y\Phi(x_{k+1},y_k;\zeta_k), y-y_k \rangle + \frac{1}{2\eta_y} \Vert y-y_k \Vert^2.
    \end{aligned}
  \end{equation*}
  The term $\langle \nabla_y\Phi(x_{k+1},y_k;\zeta_k), y_{k+1}-y_k \rangle$ is problematic, because the right~hand side of the inner product is not measurable with respect to the sigma algebra generated by past iterates $\mathcal{F}_k := \sigma\{x_{k+1},\dots, x_1, y_k, \dots, y_1\}$, so we insert and subtract $\langle \nabla_y\Phi(x_{k+1},y_k) , y_{k+1}-y_k\rangle$
  Now, using Young's inequality we estimate the resulting inner product
  \begin{multline*}
    \langle \nabla_y\Phi(x_{k+1},y_k;\zeta_k) - \nabla_y\Phi(x_{k+1},y_k), y_{k+1}-y_k \rangle \\
    \le \eta_y\Vert \nabla_y\Phi(x_{k+1},y_k;\zeta_k) - \nabla_{y}\Phi(x_{k+1},y_k) \Vert^2 + \frac{1}{4\eta_y} \Vert y_{k+1} - y_k \Vert^2.
  \end{multline*}
  Combining the above two inequalities for $y=y^*_m$ with $1\le m\le k$ and taking the expectation together with the bounded variance assumption~\eqref{eq:bounded-variance} gives
  \begin{equation}
    \label{eq:ab1}
    \begin{aligned}
    \MoveEqLeft \Exp{\Expcond{\langle \nabla_{y}\Phi(x_{k+1},y_k;\xi_k), y^*_m-y_k \rangle}{\mathcal{F}_k}}  + \Exp{h(y_{k+1})+ \frac{1}{2\eta_y} \Vert y^*_m-y_{k+1} \Vert^2}\\
    &\le  \Exp{\langle \nabla_{y}\Phi(x_{k+1},y_k), y_{k+1}-y_k \rangle - \frac{1}{4\eta_y} \Vert y_{k+1} -y_k \Vert^2}
    \\
    &\quad+ \eta_y \sigma^2+ \Exp{h(y^*_m) +\frac{1}{2\eta_y} \Vert y^*_m-y_k \Vert^2}.
    \end{aligned}
  \end{equation}
  From the descent lemma (in ascent form) and the fact that $ \eta_y \le  \nicefrac{1}{2L_{\nabla\Phi}}$ we have
  \begin{equation*}
    \Phi(x_{k+1}, y_k) + \langle y_{k+1} - y_k, \nabla_y\Phi(x_{k+1}, y_k)\rangle - \frac{1}{4\eta_y} \Vert y_{k+1} - y_k \Vert^2 \le \Phi(x_{k+1}, y_{k+1}).
  \end{equation*}
  We plug the above inequality into~\eqref{eq:ab1}, make use of the concavity and add $f(x_{k+1})$ on both sides to deduce the statement of the lemma.
\end{proof}

We can now use the previous lemma to estimate $\hat{\Delta}_k$.

\begin{lemma}%
  \label{lem:stoch-estimate-Delta}
  For all $1 \le m \le k$, we have that
  \begin{equation}
    \label{eq:stoch-estimate-Delta}
    \begin{aligned}
      \MoveEqLeft \hat{\Delta}_k \le 2\eta_x L(L_f+L+\sigma)(k-m) + \frac{1}{2\eta_y} \Exp{\Vert y_{k-1} - y^*_m \Vert^2 - \Vert y_k - y^*_m \Vert^2} + \eta_y\sigma^2.
    \end{aligned}
  \end{equation}
\end{lemma}
\begin{proof}
  Let the numbers $1 \le m \le k$ be fixed.
  Starting from the definition of $\hat{\Delta}_k$, we add~\eqref{eq:stoch-gradient-ascent-estimation} to obtain
  \begin{equation}
    \label{eq:diff-function-values-stoch}
    \begin{aligned}
      \hat{\Delta}_k &= \Exp{\psi(x_k) - \Psi(x_k,y_k)} \\
      &\le \Exp{\psi(x_k) - \Psi(x_k, y^*_m)} + \frac{1}{2\eta_y} \Big(\Exp{\Vert y^*_m - y_{k-1} \Vert^2 - \Vert y^*_m - y_k \Vert^2} \Big)+ \eta_y\sigma^2.
    \end{aligned}
  \end{equation}
  As in~\eqref{eq:plus-minus-functionval} we deduce that
  \begin{equation*}
    \Psi(x_k, y^*_k) - \Psi(x_k, y^*_m) \le \Phi(x_k,y_k^*) -\Phi(x_m,y_k^*) +\Phi(x_m,y_m^*) -\Phi(x_k,y_m^*).
  \end{equation*}
  Together with the $L$-Lipschitz continuity of $\Phi(\cdot, y)$ and~\eqref{eq:bound-prox-f} we estimate for $k\ge m$ as in~\eqref{eq:fun-val-same-y}, but using $\Exp{\Vert g_l^\xi \Vert}\le L + \sigma$,
  \begin{equation*}
    \begin{aligned}
      \Exp{\Phi(x_k, y^*_k) - \Phi(x_m, y^*_k)}
      &\le \eta_x L\Big(L_f+\sqrt{L^2+\sigma^2}\Big)(k-m).
    \end{aligned}
  \end{equation*}
  Plugging all of these into~\eqref{eq:diff-function-values-stoch} gives the statement of the lemma.
\end{proof}

In order to estimate the summation of $\hat{\Delta}_k$ we will use the same trick as in the deterministic setting and sum over it in blocks, where the size $B$ of these blocks will divide the total number of iterations $K$.

\begin{lemma}%
  \label{lem:stoch-estimate-sum-of-delta}
  We have that for all $K\ge1$
  \begin{equation}
    \label{eq:stoch-estimate-sum-of-delta}
    \frac{1}{K} \sum_{k=0}^{K-1} \hat{\Delta}_k \le \eta_x L(L+L_f+\sigma) B + \frac{D_h^2}{2\eta_y B} + \eta_y\sigma^2 + \frac{\Delta_0}{K}.
  \end{equation}
\end{lemma}
\begin{proof}
  We proceed as in Lemma~\ref{lem:estimate-sum-of-delta}. By using Lemma~\ref{lem:stoch-estimate-Delta} we obtain
  $j>0$ and $m=jB$ we have that
  \begin{equation}
    \label{eq:stoch--sum-over-any-block}
    \begin{aligned}
      \sum_{k=jB}^{(j+1)B-1} \hat{\Delta}_k 
      &\le \eta_x L(L+L_f+\sigma) B^2 + \frac{1}{2\eta_y}D_h^2 + B\eta_y \sigma^2.
    \end{aligned}
  \end{equation}
  For $j=0$ we use $m=1$ and do not estimate $\Delta_0$ but leave it there.
  Plugging~\eqref{eq:stoch--sum-over-any-block} into~\eqref{eq:split--summation-into-blocks} gives the statement of the lemma.
\end{proof}

Now we can prove the convergence result for the stochastic algorithm.

\begin{proof}[Proof of Theorem~\ref{thm:alternating-rate-stoch}]
  We sum up the inequality of Lemma~\ref{lem:stoch-descent-on-moreau} to deduce that
  \begin{equation*}
    \Exp{\psi_\lambda(x_{K})} \le \psi_\lambda(x_0) + 2\eta_x\rho\sum_{k=0}^{K-1}\hat{\Delta}_k - \frac{\eta_x}{2}\sum_{k=0}^{K-1}\Exp{\Vert \nabla \psi_\lambda(x_k) \Vert^2} + 4K\rho\eta_x^2(L^2+\sigma^2).
  \end{equation*}
  Thus, by dividing by $K$ and $\eta_x$ yields
  \begin{equation*}
    \frac{1}{K}\sum_{k=0}^{K-1} \Exp{\Vert\nabla \psi_\lambda(x_k)\Vert^2}  \le \frac{2\Delta^*}{\eta_x K} + \frac{4\rho}{K}\sum_{k=0}^{K-1}\hat{\Delta}_k + 8\rho\eta_x(L^2+\sigma^2).
  \end{equation*}
  Now we plug in~\eqref{eq:stoch-estimate-sum-of-delta} to obtain
  \begin{equation*}
    \begin{aligned}
      \MoveEqLeft\frac{1}{K}\sum_{k=0}^{K-1} \Exp{\Vert\nabla \psi_\lambda(x_k)\Vert^2} \\
      &\le \frac{2\Delta^*}{\eta_x K} + 4\rho\Big( \eta_x L(L+L_f+\sigma) B  + \frac{D_h^2}{2\eta_y B} + \eta_y\sigma^2 \Big)  + 4\rho\frac{\Delta_0}{K}+ 8\rho\eta_x (L^2+\sigma^2).
    \end{aligned}
  \end{equation*}
  With the block size $B=D_h \sqrt{1/(\eta_x\eta_y L(L+L_f+\sigma))}$ we have that
  \begin{equation*}
    \eta_x L(L+L_f+\sigma) B  + \frac{D_h^2}{2\eta_y B} + \eta_y\sigma^2 = \sqrt{\frac{\eta_x}{\eta_y}}\sqrt{L(L+L_f+\sigma) }D_h + \eta_y \sigma^2
  \end{equation*}
  Via the step size choice presented in the theorem we obtain the desired complexity.
\end{proof}

\subsection{Alternating vs simultaneous}%
\label{sub:altvssim-non-strongly}

Although we are not able to show improved rates for the alternating version of GDA in this setting, we would still like to point out some improvements in the constants which otherwise might go unnoticed since the statements are quite technical.

In the nonconvex-concave setting the main descent type property we focus on can be seen in Lemma~\ref{lem:descent-on-moreau} and is given by
\begin{equation*}
  g_\lambda(x_{k+1}) \le g_\lambda(x_k) + 2\rho\eta_x\Delta_k - \frac12\eta_x \Vert \nabla g_\lambda(x_k) \Vert^2  + 4\rho\eta_x^2L^2,
\end{equation*}
From this it clear the only troublesome part is the estimation of $\Delta_k= \psi(x_k)-\Psi(x_k,y_k)$ (to be precise we, we need to estimate the sum of $\Delta_k$ after telescoping).
While we obtain
\begin{equation*}
    \Delta_k \le \Phi(x_k,y_k) - \Phi(x_m,y_k^*)+\Phi(x_m, y_m^*) - \Phi(x_k,y_m^*) + \frac{1}{2 \eta_y}(\Vert y - y_{k-1} \Vert^2 - \Vert y - y_k \Vert^2)
\end{equation*}
in~\cite{lin2020gradient} the same estimate is obtained for simultaneous GDA  plus an additional term
\begin{equation*}
  \Phi(x_k,y_k) -\Phi(x_{k-1}, y_{k-1}),
\end{equation*}
see~\cite[Lemma D.4]{lin2020gradient}. While we do not have information about the sign of this term it is clear that its absence is preferable as it needs to be estimated after telescoping and averaging via
\begin{equation*}
  \frac{1}{K} (\Phi(x_K, y_K) - \Phi(x_0, y_0)) \le \eta_x L^2 + \Delta_0.
\end{equation*}
Looking at the statement of Lemma~\ref{lem:estimate-Delta} we see, however, that factors of both of these terms already appear in the final statement due to other estimations which is why their appearance gets lost in the big O notation.

\section{Nonconvex-strongly concave objective}%
\label{sec:strongly-concave}

By requiring in addition to the assumptions of Section~\ref{sec:non-strongly-concave} \emph{strong} convexity in the second component and smoothness of the coupling function in $x$, we can drop any assumption about Lipschitz continuity and will be able to deduce the max function $\phi$ is smooth with Lipschitz continuous gradient (making it weakly convex). For the technical details see the following assumptions.

\begin{assumption}%
  \label{ass:strongly-concave}
  Let $\Phi$ be $L_{\nabla\Phi}$-smooth uniformly in both components and concave in the second one. The regularizers $f$ and $-h$ are proper, l.s.c.\ and convex.
  Additionally, either $\Phi$ is $\mu$-strongly concave in the second component, uniformly in the first one, or $-h$ is $\mu$-strongly concave.
\end{assumption}


\paragraph{Notation}
In Proposition~\ref{thm:extended-danskin-2} we will show that under the above assumptions $\phi= \max_{y\in\R^n} \{\Phi(\cdot, y) - h(y)\}$ is $L_{\nabla \phi}$-smooth, with $L_{\nabla \phi} = (1+\kappa)L_{\nabla\Phi}$, for $\kappa:= \max\{\nicefrac{L_{\nabla\Phi}}{\mu},1\}$ denoting the \emph{condition number}. In the setting without regularizers, where the strong concavity arises from $\Phi$ it is well known that $\mu\le L_{\nabla\Phi}$ and therefore $1\le \nicefrac{L_{\nabla\Phi}}{\mu}$ (the standard definition of the condition number). If the strong concavity stems from the regularizers $h$ this is no longer true and $\nicefrac{L_{\nabla\Phi}}{\mu}$ might be smaller than $1$ which would lead to tedious case distinctions, which is why we adapt the definition of the condition number in order to provide a unified analysis.
Additionally, the solution set $Y(x)$ defined in~\eqref{eq:solution-set-mapping} consists only of a single element which we will denote by $y^*(x)$. We denote the quantity $\delta_k:= \Vert y_k-y^*_k \Vert^2$, measuring the distance between the current strategy of the second player and her best response according to the current strategy of the first player.

\subsection{Properties of the max function}%
\label{sub:properties-max-function-strongly-concave}

In the following we will show the smoothness of $\phi$, as well as the fact that the solution map fulfills a strong Lipschitz property.

\begin{lemma}[Lipschitz continuity of the solution mapping]%
  \label{lem:solution-map-is-lipschitz}
  The solution map $y^*:\R^d \to \R^n$ which fulfills $\Gamma(x,y^*(x)) = \max_{y\in\R^n} \Gamma(x,y)$ for all $x\in\R^d$ is well defined and $\kappa$-Lipschitz where $\kappa=\max\{\nicefrac{L_{\nabla\Phi}}{\mu},1\}$.
\end{lemma}
\begin{proof}
  Let $x,x' \in \R^d$ be fixed. From the optimality condition we deduce that
  \begin{equation*}
    \nabla_y\Phi(x',y^*(x')) - \nabla_y \Phi(x,y^*(x')) \in \partial h(y^*(x')) - \nabla_y \Phi(x,y^*(x')).
  \end{equation*}
  Thus by the strong monotonicity of $\partial h - \nabla_y \Phi(x,\cdot)$ we obtain
  \begin{equation*}
    \begin{aligned}
      \mu \Vert y^*(x)-y^*(x') \Vert^2 &\le \langle y^*(x)-y^*(x'), \nabla_y\Phi(x,y^*(x'))-\nabla_y \Phi(x',y^*(x')) \rangle\\
      &\le \Vert y^*(x)-y^*(x') \Vert L_{\nabla\Phi}\Vert x-x' \Vert.
    \end{aligned}
  \end{equation*}
  The statement of the lemma follows.
\end{proof}

\begin{proposition}[Smoothness of the max function]%
  \label{thm:extended-danskin-2}
  Let Assumption~\ref{ass:strongly-concave} hold true. Then, $\phi$ is smooth and its gradient is given by
  \begin{equation*}
    \nabla \phi(x) = \nabla_x \Phi(x,y^*(x))
  \end{equation*}
  and is therefore $L_{\nabla\Phi}(1 +\kappa)$-Lipschitz.
\end{proposition}
\begin{proof}
  Following Proposition~\ref{thm:extended-danskin}, we define the quantities $\tilde{\phi}, \tilde{\Phi}$ and $\tilde{\Gamma}$ as there using $\rho=L_{\nabla\Phi}$.
  Let $x,v \in \R^d$, $\alpha_k \downarrow 0$ and $x^k:=x + \alpha_k v$ for any $k \geq 0$. Further, let be $y^k = y^*(x^k)$ for any $k \geq 0$. Then, by the Lipschitz continuity of $y^*(\cdot)$, see Lemma~\ref{lem:solution-map-is-lipschitz}, $\lim_{k\to\infty} y^k = y^*(x)$.
  In addition, for any $v \in \R^d$ and all $k \geq 0$,
  \begin{align*}
    \tilde{\phi}'(x;v) &\leq  \frac{\tilde{\phi}(x^k)- \tilde{\phi}(x)}{\alpha_k}  \leq \frac{\tilde{\Gamma}(x^k,y^k) - \tilde{\Gamma}(x,y^k)}{\alpha_k}\\
                    &=  \frac{\tilde{\Phi}(x^k, y^k)- \tilde{\Phi}(x,y^k)}{\alpha_k} \leq
                      [\tilde{\Phi}(\cdot, y^k)]'(x+\alpha_k v;v) = \langle \nabla_x \tilde{\Phi}(x^k,y^k), v\rangle .
  \end{align*}
  Since the gradient of $\tilde{\Phi}$ is continuous, this implies by letting $k \rightarrow +\infty$ that
  \begin{equation*}
    \label{eq:reverse-gradient}
    \tilde{\phi}'(x;v) \leq \langle \nabla_x \tilde{\Phi}(x,y^*(x)), v\rangle, \quad \forall v \in \R^d,
  \end{equation*}
  which, together with~\eqref{eq1}, yields that~\eqref{eq:reverse-gradient} holds with equality.
  The fact that the gradient of $\phi$ is Lipschitz continuous follows, with $y^*=y^*(x)$ and $\bar{y}^{*}=\bar{y}^*(\bar{x})$, from
  \begin{equation*}
    \begin{aligned}
       \Vert \nabla\phi(x)-\nabla\phi(\bar{x}) \Vert &\le \Vert \nabla_x\Phi(x,y^*)-\nabla_x\Phi(\bar{x},y^*) \Vert + \Vert \nabla_x\Phi(\bar{x},y^*)-\nabla_x\Phi(\bar{x},\bar{y}^*) \Vert \\
      &\le L_{\nabla\Phi} \Vert x-\bar{x} \Vert + L_{\nabla\Phi}\Vert y^*-\bar{y}^* \Vert \le (L_{\nabla\Phi} + L_{\nabla\Phi} \kappa) \Vert x-\bar{x} \Vert,
    \end{aligned}
  \end{equation*}
  together with the claimed constant.
\end{proof}

\subsection{Deterministic setting}%
\label{sub:deterministic-strongly-concave}

For this section Algorithm~\ref{alg:alternatinggda} reads as
\begin{equation}%
  \label{eq:alt-gda-strongly}
  (\forall k \ge 0)
  \left\lfloor
    \begin{array}{l}
      x_{k+1} = \prox{\eta_x f}{x_k - \eta_x \nabla_x\Phi(x_k,y_k)} \\
      y_{k+1} = \prox{\eta_y h}{y_k + \eta_y \nabla_y\Phi(x_{k+1},y_k)}.
    \end{array}\right.
\end{equation}
We start with the main convergence result of this section.
\begin{theorem}%
  \label{thm:rate-strongly-concave}
  Let Assumption~\ref{ass:strongly-concave} and~\ref{ass:max-function-lower-bounded} hold. For algorithm~\ref{eq:alt-gda-strongly} with step size $\eta_y=\nicefrac{1}{L_{\nabla\Phi}}$ and $\eta_x=\nicefrac{1}{(3{(\kappa+1)}^2L_{\nabla\Phi})}$ the number of gradient evaluations $K$ required, using the notation $\Delta^* = \psi(x_0)- \inf_{x\in\R^d} \psi(x)$, is
  \begin{equation*}
    \mathcal{O}\left(\frac{\kappa L_{\nabla\Phi}}{\epsilon^2} \max\{\kappa \Delta^*, L_{\nabla\Phi} \Vert y^*(x_0)- y_0 \Vert^2\}\right),
  \end{equation*}
  to visit an $\epsilon$-stationary point such that $\min_{1\le k\le K}\dist\big(-\nabla\phi(x_k), \partial f(x_k)\big) \le \epsilon$.
\end{theorem}

Before we start with the first lemma, let us recall that $\delta_k = \Vert y_k - y_k^* \Vert^2$ denotes for all $k\ge0$ the squared distance  between the current iterate $y_k$ and the maximizing argument $y_k^* = \argmax_{y} \Psi(x_k,y)$.

\begin{lemma}%
  \label{lem:dist-iter-bounds-optimality}
  There exists a sequence ${(w_k)}_{k\ge1}$ such that $w_k \in (\partial f+\nabla\phi)(x_k)$ and its norm can be bounded for all $k \ge 0$ by
  \begin{equation*}
      \frac12 \eta_x \Vert w_{k+1} \Vert^2 \le \psi(x_k)-\psi(x_{k+1}) + \frac12\Big(L_{\nabla\phi} + 2L_{\nabla\phi}^2\eta_x - \frac{1}{\eta_x} \Big) \Vert x_k-x_{k+1} \Vert^2 + \eta_x L_{\nabla\Phi}^2 \delta_k.
  \end{equation*}
\end{lemma}
\begin{proof}
  Let $k\ge0$ be arbitrary but fixed. From the optimality condition of the proximal operator we deduce by adding $\nabla \phi(x_{k+1})$ on both sides
  \begin{equation*}
    w_{k+1}:= \frac{1}{\eta_x}(x_k-x_{k+1}) + \nabla\phi(x_{k+1}) - \nabla_x\Phi(x_k,y_k) \in \partial f(x_{k+1})+\nabla\phi(x_{k+1}),
  \end{equation*}
  as claimed. In order to prove the bound on $\Vert w_{k+1} \Vert$ we proceed as follows:
  \begin{equation}
    \label{eq:estimate-wk}
    \begin{aligned}
    \Vert w_{k+1} \Vert^2 &= \eta_x^{-2}\Vert x_k-x_{k+1} \Vert^2 + 2\eta_x^{-1} \langle x_k-x_{k+1}, \nabla\phi(x_{k+1})-\nabla_x\Phi(x_k,y_k) \rangle \\
    & \quad+ \Vert \nabla\phi(x_{k+1})-\nabla_x\Phi(x_k,y_k) \Vert^2.
    \end{aligned}
  \end{equation}
  The smoothness of $\phi$ implies via the descent lemma that
  \begin{equation}
    \label{eq:ascent-lemma-phi}
    \phi(x_{k+1})+\langle \nabla\phi(x_{k+1}), x_k-x_{k+1} \rangle - \frac{L_{\nabla\phi}}{2} \Vert x_{k+1}-x_k \Vert^2 \le \phi(x_k).
  \end{equation}
  Since the proximal operator minimizes a $\nicefrac{1}{\eta_x}$-strongly convex function we have that
  \begin{equation*}
    \begin{aligned}
      \MoveEqLeft f(x_{k+1})+ \langle \nabla_x\Phi(x_k,y_k),x_{k+1}-x_k \rangle + \frac{1}{\eta_x}\Vert x_{k+1}-x_k \Vert^2 \le f(x_k)
    \end{aligned}
  \end{equation*}
  Adding this inequality to~\eqref{eq:ascent-lemma-phi} we deduce that
  \begin{equation}
    \label{eq:descent-grad-kplus1}
    \begin{aligned}
      \langle \nabla\phi(x_{k+1})-\nabla_x\Phi(x_k,y_k), x_k-x_{k+1} \rangle &\le \psi(x_k) - \psi(x_{k+1}) \\
      &\quad+\frac12\left( L_{\nabla\phi}-\frac{2}{\eta_x} \right) \Vert x_{k+1}-x_k \Vert^2.
    \end{aligned}
  \end{equation}
  Lastly, by the Young inequality
  \begin{equation*}
    \begin{aligned}
      \Vert \nabla\phi(x_{k+1})-\nabla_x\Phi(x_k,y_k) \Vert^2 &= \Vert \nabla\phi(x_{k+1})-\nabla\phi(x_k)+\nabla\phi(x_k)-\nabla_x\Phi(x_k,y_k) \Vert^2 \\
      &\le 2L_{\nabla\phi}^2\Vert x_{k+1}-x_k \Vert^2 + 2L_{\nabla\Phi}^2\delta_k.
    \end{aligned}
  \end{equation*}
  Plugging~\eqref{eq:descent-grad-kplus1} and~\eqref{eq:ascent-lemma-phi} into~\eqref{eq:estimate-wk} yields the desired statement.
\end{proof}

In the next lemma it remains to bound the gap between the current iterate and the maximizing argument of the second component $\delta_k = \Vert y^*_k -y_k \Vert^2$.

\begin{lemma}%
  \label{lem:recursion-delta}
  We have that for all $k\ge0$, then
  \begin{equation*}
    \delta_{k+1} \le \left(1 - \frac{1}{2\kappa} \right)\delta_k + \kappa^3 \Vert x_{k+1}-x_k \Vert^2.
  \end{equation*}
\end{lemma}
\begin{proof}
  Let $k\ge0$ be fixed.
  From the definition of $y_{k+1}$, see~\eqref{eq:alt-gda-strongly}, and the fact that $y_{k+1}^*$ is a fixed point of the proximal-gradient operator, we deduce
  \begin{equation*}
    \begin{aligned}
      \delta_{k+1} 
      = \Vert \prox{\eta_y h}{y^*_{k+1} + \eta_y \nabla_y \Phi(x_{k+1},y^*_{k+1})} - \prox{\eta_y h}{y_{k}+ \eta_y\nabla_y \Phi(x_{k+1},y_k)} \Vert^2.
    \end{aligned}
\end{equation*}
  If \(\Phi\) is strongly concave in its second component we can use the nonexpansiveness of the proximal operator
  and~\cite[Theorem 2.1.11]{nesterov-introductory}, which states
  \begin{equation}
    \label{eq:nesterov-bound}
    \begin{aligned}
      \MoveEqLeft \langle \nabla_y\Phi(x_{k+1},y_{k+1}^*)-\nabla_y\Phi(x_{k+1},y_k),y_{k+1}^*-y_k \rangle \\ &\le -\frac{\mu L_{\nabla\Phi}}{\mu+L_{\nabla\Phi}}\Vert y_{k+1}^*-y_k \Vert^2 - \frac{1}{\mu+L_{\nabla\Phi}} \Vert \nabla_y\Phi(x_{k+1},y_{k+1}^*)-\nabla_y\Phi(x_{k+1},y_k) \Vert^2
  \end{aligned}
  \end{equation}
  to conclude
  \begin{equation*}
    \begin{aligned}
      \delta_{k+1} &\le \Vert y^*_{k+1} + \eta_y\nabla_y \Phi(x_{k+1},y^*_{k+1}) - y_{k}- \eta_y\nabla_y \Phi(x_{k+1},y_k) \Vert^2 \\
      &= \Vert y_{k+1}^* -y_k \Vert^2 + 2\eta_y\langle y_{k+1}^* -y_k, \nabla_y \Phi(x_{k+1},y^*_{k+1})-\nabla_y \Phi(x_{k+1},y_k) \rangle \\
      &\quad+ \eta_y^2\Vert \nabla_y \Phi(x_{k+1},y^*_{k+1})-\nabla_y \Phi(x_{k+1},y_k) \Vert^2\\
      &\overset{~\eqref{eq:nesterov-bound}}{\le}\left( \frac{\kappa-1}{\kappa+1} \right) \Vert y_{k+1}^* -y_k \Vert^2 \le q\Vert y_{k+1}^* -y_k \Vert^2
    \end{aligned}
  \end{equation*}
  with $q:= {\left(\frac{\kappa}{\kappa+1}\right)}^2$, where we used that $\eta_y=\nicefrac{1}{L_{\nabla\Phi}}$.
  If on the other hand $-h$ is strongly concave we can use the fact that the proximal operator (of $h$) is even a contraction, see~\cite[Proposition 25.9 (i)]{bc}, to deduce that \(\delta_{k+1} \le q\Vert y_{k+1}^* -y_k \Vert^2 \).
  Therefore, in either case \(\delta_{k+1} \le q\Vert y_{k+1}^* -y_k \Vert^2 \).
  Using this, the triangle inequality and Young's inequality, we have
  \begin{align}
    \delta_{k+1} &\le q \Vert y^*_{k+1} - y_k \Vert^2 \nonumber \le q {\Big( \Vert y^*_{k+1} - y^*_k \Vert + \Vert y^*_k-y_k \Vert\Big)}^2\nonumber \\
            &\le q\left(1 + \frac{3 \kappa^2-1}{2 \kappa^3} \right)\underbrace{\Vert y^*_k-y_k \Vert^2}_{=\delta_k} + q\left(1+\frac{2 \kappa^3}{3 \kappa^2-1}\right)\Vert y^*_{k+1}-y^*_k \Vert^2\nonumber \\
            &\le \left(1 - \frac{1}{2\kappa}\right) \delta_k + \kappa \Vert y^*_{k+1}-y^*_k \Vert^2. \label{eq:estimate-delta1}
  \end{align}
  Due to the $\kappa$-Lipschitz continuity of $y^*(\cdot)$ we have that $\Vert y^*_{k+1} - y^*_k \Vert \le \kappa \Vert x_{k+1}-x_k \Vert$, which finishes the proof.
\end{proof}
Now we can bound the sum of $\delta_k$.
\begin{lemma}%
  \label{lem:bound-sum-delta}
  We have that, for all $K\ge1$
  \begin{equation*}
    \sum_{k=0}^{K-1} \delta_k \le 2\kappa\delta_0 + 2\kappa^4 \sum_{k=0}^{K-1}\Vert x_{k+1}-x_k \Vert^2.
  \end{equation*}
\end{lemma}
\begin{proof}
  By recursively applying the previous lemma we obtain for $k\ge1$
  \begin{equation*}
    \delta_k \le {\left(1-\frac{1}{2 \kappa}\right)}^k\delta_0 + \kappa^3\sum_{j=0}^{k-1}{\left(1-\frac{1}{2\kappa}\right)}^{k-j-1}\Vert x_{j+1}-x_j \Vert^2.
  \end{equation*}
  Now we sum this inequality from $k=1$ to $K-1$ and add $\delta_0$ on both sides to deduce
  \begin{equation*}
    \sum_{k=0}^{K-1} \delta_k \le 2\kappa\delta_0 + 2\kappa^4 \sum_{k=0}^{K-1}\Vert x_{k+1}-x_k \Vert^2,
  \end{equation*}
  where we used that
  \begin{equation}
    \label{eq:product}
    \sum_{k=1}^{K-1}\sum_{j=0}^{k-1} {\left(1-\frac{1}{2 \kappa}\right)}^{k-1-j} \Vert x_{j+1}-x_j \Vert^2 \le \sum_{j=0}^{K-1} {\left(1-\frac{1}{2 \kappa}\right)}^j \left(\sum_{k=0}^{K-1} \Vert x_{k+1}-x_k \Vert^2\right),
  \end{equation}
  and $\sum_{j=0}^{\infty} {(1-{(2\kappa)}^{-1})}^j= 2\kappa$.
\end{proof}
We can now put the pieces together.

\begin{proof}[Proof of Theorem~\ref{thm:rate-strongly-concave}]
  Summing up the inequality of Lemma~\ref{lem:dist-iter-bounds-optimality} from $k=0$ to $K-1$ and applying Lemma~\ref{lem:bound-sum-delta} we deduce that
  \begin{equation*}
    \begin{aligned}
      \frac12\eta_x \sum_{k=1}^{K} \Vert w_k \Vert^2 &\le \psi(x_0)-\psi(x_{K})  + 2L_{\nabla\Phi}^2\eta_x \kappa\delta_0\\
      &\quad+ \frac12\Big(L_{\nabla\phi}+2L_{\nabla\phi}^2\eta_x - \frac{1}{\eta_x} + 2\kappa^4L_{\nabla\Phi}^2\eta_x \Big)\sum_{k=0}^{K-1} \Vert x_{k+1}-x_k \Vert^2.
    \end{aligned}
  \end{equation*}
  With the step size $\eta_x= 1/(3{(\kappa+1)}^2 L_{\nabla\Phi})$ it follows that
  \begin{equation*}
    \begin{aligned}
      L_{\nabla\phi}+2L_{\nabla\phi}^2\eta_x - \frac{1}{\eta_x} + 2\kappa^4L_{\nabla\Phi}^2\eta_x
      &\le - \frac23 {(\kappa+1)}^2 L_{\nabla\Phi} \le 0,
    \end{aligned}
  \end{equation*}
  which concludes the proof.
\end{proof}

\subsection{Stochastic setting}%
\label{sub:stoch-strongly-convex}

For the purpose of this section Algorithm~\ref{alg:alternatinggda} reads
\begin{equation}
  \label{eq:alt-gda-strongly-stoch}
  (\forall k \ge 0)
  \left\lfloor
    \begin{array}{l}
      x_{k+1} = \prox{\eta_x f}{x_k - \eta_x G_x} \\
      y_{k+1} = \prox{\eta_y h}{y_k + \eta_y G_y},
    \end{array}\right.
\end{equation}
for the minibatch gradient estimators, given by $G_x = \frac{1}{M}\sum_{i=1}^{M}\nabla_x\Phi(x_k,y_k;\xi^i_k)$ and $G_y = \frac{1}{M}\sum_{i=1}^{M}\nabla_y\Phi(x_{k+1},y_k;\zeta^i_k)$.

\begin{theorem}%
  \label{thm:rate-strongly-concave-stoch}
  Let in addition to the assumptions of Theorem~\ref{thm:rate-strongly-concave} also the two properties of the gradient estimator Assumption~\ref{ass:unbiased} and~\ref{ass:bounded-variance} hold true. For algorithm~\eqref{eq:alt-gda-strongly-stoch} with step size $\eta_y=\nicefrac{1}{L_{\nabla\Phi}}$ and $\eta_x=\nicefrac{1}{(4{(1+\kappa)}^2 L_{\nabla\Phi})}$ and batch size $M=\mathcal{O}(\kappa \sigma^2\epsilon^{-2})$ the number of stochastic gradient evaluations $K$ required is
  \begin{equation*}
    \begin{aligned}
      \mathcal{O}\left(\frac{\sigma^2 \kappa^2L_{\nabla\Phi}}{\epsilon^4} \max\{\kappa \Delta^*, L_{\nabla\Phi} \Vert y^*(x_0)-y_0 \Vert^2\}\right),
    \end{aligned}
  \end{equation*}
  such that $\min_{1\le k\le K}\E[\dist\big(-\nabla\phi(x_k), \partial f(x_k)\big)] \le \epsilon$, i.e.\
  to visit an $\epsilon$-stationary point in expectation, where $\Delta^* = \psi(x_0)- \inf_{x\in\R^d} \psi(x)$.
\end{theorem}


\begin{lemma}%
  \label{lem:dist-iter-bounds-optimality-stoch}
  There exists a sequence ${(w_k)}_{k\ge1}$ such that $w_k \in (\partial f+\nabla\phi)(x_k)$ and its norm can be bounded for all $k\ge0$ by
  \begin{equation*}
    \begin{aligned}
      \frac12 \eta_x \Exp{\Vert w_{k+1} \Vert^2} &\le \Exp{\psi(x_k)-\psi(x_{k+1})}+ \eta_x L_{\nabla\Phi}^2 \Exp{\delta_k} + \eta_x \frac{\sigma^2}{M} \\
      &\quad+ \frac12\Big(L_{\nabla\phi} + 3L_{\nabla\phi}^2\eta_x - \frac{1}{\eta_x} \Big) \Exp{\Vert x_k-x_{k+1} \Vert^2}.
    \end{aligned}
  \end{equation*}
\end{lemma}
\begin{proof}
  Let $k\ge0$.
  From the proximal operator we deduce that
  \begin{equation*}
    0 \in \partial f(x_{k+1}) + G_x + \frac{1}{\eta_x}(x_{k+1}-x_k).
  \end{equation*}
  Thus, we define $w_{k+1}$ such that we immediately obtain the desired inclusion
  \begin{equation*}
    \begin{aligned}
      w_{k+1}:= \frac{1}{\eta_x}(x_k-x_{k+1}) + \nabla\phi(x_{k+1}) - G_x \in \partial f(x_{k+1})+\nabla\phi(x_{k+1}).
    \end{aligned}
  \end{equation*}
  In order to bound $w_{k+1}$ we consider
  \begin{equation}
    \label{eq:estimate-wk-stoch}
    \begin{aligned}
      \MoveEqLeft[1] \Vert w_{k+1} \Vert^2 \\
      &= \frac{1}{\eta_x^2} \Vert x_k-x_{k+1} \Vert^2 + \frac{2}{\eta_x} \langle x_k-x_{k+1}, \nabla\phi(x_{k+1})-G_x \rangle + \Vert \nabla\phi(x_{k+1})-G_x \Vert^2.
    \end{aligned}
  \end{equation}
  Using the analogous statement to~\eqref{eq:descent-grad-kplus1}
  \begin{equation*}
    \begin{aligned}
      \MoveEqLeft \Exp{\Vert \nabla\phi(x_{k+1}) - G_x \Vert^2} =\Exp{ \Vert \nabla\phi(x_{k+1})-\nabla_x\Phi(x_k,y_k)+\nabla_x\Phi(x_k,y_k)-G_x \Vert^2} \\
      &= \Exp{\Vert \nabla\phi(x_{k+1})-\nabla\phi(x_k)+\nabla\phi(x_k)-\nabla_x\Phi(x_k,y_k) \Vert^2} \\
      &\quad+2 \Exp{\langle \nabla\phi(x_{k+1})-\nabla\phi(x_k), \nabla_x\Phi(x_k,y_k) - G_x\rangle} + \Exp{\Vert \nabla_x\Phi(x_k,y_k)-G_x \Vert^2}\\
      &\quad+ 2\underbrace{\Exp{\langle \nabla\phi(x_k)-\nabla_x\Phi(x_k,y_k),\nabla_x\Phi(x_k,y_k)-G_x\rangle}}_{=0} \\
      &\le3L_{\nabla\phi}^2\Exp{\Vert x_{k+1}-x_k \Vert^2} + 2L_{\nabla\Phi}^2\Exp{\delta_k} + 2 \frac{\sigma^2}{M}
    \end{aligned}
  \end{equation*}
  in~\eqref{eq:estimate-wk-stoch} yields the desired statement.
\end{proof}

In the next lemma it remains to bound $\delta_k$.

\begin{lemma}%
  \label{lem:recursion-delta-stoch}
  We have that for all $k\ge0$
  \begin{equation*}
    \Exp{\delta_{k+1}} \le {\left(1-\frac{1}{2\kappa}\right)} \Exp{\delta_k} + \kappa^3 \Vert x_{k+1}-x_k \Vert^2 + \frac{\sigma^2}{M L_{\nabla\Phi}^2}.
  \end{equation*}
\end{lemma}
\begin{proof}
  Let $k\ge0$ be fixed.
  We first consider the case where $\Phi$ is strongly concave in its second component from the definition of $y_{k+1}$ (see~\eqref{eq:alt-gda-strongly-stoch}) we deduce that
  \begin{equation*}
    \label{eq:strong-concavity-contraction-stoch}
    \begin{aligned}
      \delta_{k+1} &= \Vert \prox{\eta_y h}{y^*_{k+1}+\eta_y \nabla_y\Phi(x_{k+1},y_{k+1}^*)}- \prox{\eta_y h}{y_k + \eta_y G_y} \Vert^2 \\
      &\le \Vert y^*_{k+1}+\eta_y \nabla_y\Phi(x_{k+1},y_{k+1}^*) - y_k - \eta_y G_y \Vert^2 \\
    \end{aligned}
  \end{equation*}
  and
  \begin{equation}
    \label{eq:stoch-gradient-est}
    \begin{aligned}
      \MoveEqLeft \Vert y^*_{k+1}+\eta_y \nabla_y\Phi(x_{k+1},y_{k+1}^*) - y_k - \eta_y G_y \Vert^2 \\
      &= \Vert y^*_{k+1} - y_k \Vert^2 + 2\eta_y \langle y_{k+1}^*-y_k, \nabla_y\Phi(x_{k+1},y_{k+1}^*)-\nabla_y\Phi(x_{k+1},y_k) \rangle\\
      &\quad+2\eta_y \underbrace{\langle y_{k+1}^*-y_k, \nabla_y\Phi(x_{k+1},y_k)-G_y\rangle}_{(\square)}\\
      &\quad+\eta_y^2 \Vert \nabla_y\Phi(x_{k+1},y_{k+1}^*)-\nabla_y\Phi(x_{k+1},y_k) \Vert^2 +\eta_y^2 \Vert \nabla_y\Phi(x_{k+1},y_k)- G_y \Vert^2\\
      &\quad+ \eta_y^2 \underbrace{\langle \nabla_y\Phi(x_{k+1},y_{k+1}^*)-\nabla_y\Phi(x_{k+1},y_k), \nabla_y\Phi(x_{k+1},y_k)-G_y \rangle}_{(*)}.
    \end{aligned}
  \end{equation}
  Some of the terms vanish after taking the expectation such as
  \begin{equation*}
    \Exp{(*)} = \Exp{\Expcond{(*)}{y_k,x_{k+1}}} = \Exp{0}=0=  \Exp{\Expcond{(\square)}{y_k,x_{k+1}}} =\Exp{(\square)}.
  \end{equation*}
  Using furthermore~\cite[Theorem 2.1.11]{nesterov-introductory}
  which states that
  \begin{equation}
    \label{eq:strong-cocoercive}
    \begin{aligned}
      \MoveEqLeft[1] \langle \nabla_y\Phi(x_{k+1},y_k)-\nabla_y\Phi(x_{k+1},y_{k+1}^*),y_{k+1}^*-y_k \rangle \\
      &\ge \frac{\mu L_{\nabla\Phi}}{\mu+L_{\nabla\Phi}}\Vert y_{k+1}^*-y_k \Vert^2 + \frac{1}{\mu+L_{\nabla\Phi}} \Vert \nabla_y\Phi(x_{k+1},y_{k+1}^*)-\nabla_y\Phi(x_{k+1},y_k) \Vert^2
  \end{aligned}
  \end{equation}
  results in
  \begin{equation}
    \label{eq:delta-almost-recursion}
    \Exp{\delta_{k+1}} \le \left(\frac{\kappa-1}{\kappa+1}\right) \Exp{\Vert y_{k+1}^*-y_k \Vert^2} + \frac{\sigma^2}{ML_{\nabla\Phi}^2}\le  q\Exp{\Vert y_{k+1}^*-y_k \Vert^2} + \frac{\sigma^2}{ML_{\nabla\Phi}^2},
  \end{equation}
  with $q= {\left(\frac{\kappa}{\kappa+1}\right)}^2$, where we used that $\eta_y=\nicefrac{1}{L_{\nabla\Phi}}$.
  If \(h\) is strongly concave then we use the fact that the proximal operator is a contraction, see~\cite[Proposition 23.11]{bc}, to deduce that
  \begin{equation*}
    \begin{aligned}
      \E\delta_{k+1} 
      &= \Exp{\Vert \prox{\eta_y h}{y^*_{k+1}+\eta_y \nabla_y\Phi(x_{k+1},y_{k+1}^*)}- \prox{\eta_y h}{y_k + \eta_y G_y} \Vert^2} \\
      &= q \Exp{\Vert y^*_{k+1}- y_k + \eta_y \nabla_y\Phi(x_{k+1},y_{k+1}^*) - \eta_y G_y \Vert^2} \\
      &\overset{\eqref{eq:stoch-gradient-est}}{=} q\E\Big[\Vert y^*_{k+1} - y_k \Vert^2 +
      2q\eta_y \langle y_{k+1}^*-y_k, \nabla_y\Phi(x_{k+1},y_{k+1}^*)-\nabla_y\Phi(x_{k+1},y_k)\rangle \\
      &\quad+q\eta_y^2 \Vert \nabla_y\Phi(x_{k+1},y_{k+1}^*)-\nabla_y\Phi(x_{k+1},y_k) \Vert^2 + q\eta_y^2\Vert \nabla_y\Phi(x_{k+1},y_k)- G_y \Vert^2\Big].\\
    \end{aligned}
  \end{equation*}
  Using now~\eqref{eq:strong-cocoercive} with $\mu=0$, i.e.\ the cocoercivity of the gradient, we deduce that
  \begin{equation*}
    \Exp{\delta_{k+1}} = \Exp{\Vert y^*_{k+1}- y_{k+1} \Vert^2} \le q \Vert y^*_{k+1}- y_k \Vert^2 + q\frac{\sigma^2}{ML_{\nabla\Phi}^2},
  \end{equation*}
  meaning that we concluded~\eqref{eq:delta-almost-recursion} in both cases.
  Next, using~\eqref{eq:delta-almost-recursion} and the considerations made in~\eqref{eq:estimate-delta1} we deduce that
  \begin{align*}
    \Exp{\delta_{k+1}} \le \left(1 - \frac{1}{2\kappa}\right) \Exp{\delta_k} + \kappa \Exp{\Vert y^*_{k+1}-y^*_k \Vert^2} + \frac{\sigma^2}{ML_{\nabla\Phi}^2}.
  \end{align*}
  Again, due to the $\kappa$-Lipschitz continuity of $y^*(\cdot)$ we have that $\Vert y^*_{k+1} - y^*_k \Vert \le \kappa \Vert x_{k+1}-x_k \Vert$, which finishes the proof.
\end{proof}

Now we can bound the sum of $\delta_k$.

\begin{lemma}%
  \label{lem:bound-sum-delta-stoch}
  We have that, for all $K\ge1$
  \begin{equation*}
    \sum_{k=0}^{K-1} \Exp{\delta_k} \le 2\kappa\delta_0 + 2\kappa^4 \sum_{k=0}^{K-1}\Exp{\Vert x_{k+1}-x_k \Vert^2} + 2K \frac{\kappa\sigma^2}{ML_{\nabla\Phi}^2}.
  \end{equation*}
\end{lemma}
\begin{proof}
  By recursively applying the previous lemma we obtain for $k\ge1$
  \begin{equation*}
    \Exp{\delta_k} \le {\left(1-\frac{1}{2\kappa}\right)}^k\delta_0 + \sum_{j=0}^{k-1}{\left(1-\frac{1}{2 \kappa}\right)}^{k-j-1}\left(\kappa^3\Exp{\Vert x_{j+1}-x_j \Vert^2} + \frac{\sigma^2}{M L_{\nabla\Phi}^2}\right).
  \end{equation*}
  Now we sum this inequality from $k=1$ to $K-1$ and add $\delta_0$ on both sides to deduce
  \begin{equation*}
    \sum_{k=0}^{K-1} \Exp{\delta_k} \le 2\kappa\delta_0 + 2K\frac{\kappa\sigma^2}{ML_{\nabla\Phi}^2} + 2\kappa^4 \sum_{k=0}^{K-1}\Exp{\Vert x_{k+1}-x_k \Vert^2}
  \end{equation*}
  using the considerations made in~\eqref{eq:product}.
\end{proof}

We can now put the pieces together.

\begin{proof}[Proof of Theorem~\ref{thm:rate-strongly-concave-stoch}]
  We sum up the inequality of Lemma~\ref{lem:dist-iter-bounds-optimality-stoch} from $k=0$ to $K-1$ and applying Lemma~\ref{lem:bound-sum-delta-stoch} we deduce that
  \begin{equation*}
    \begin{aligned}
       \eta_x \sum_{k=1}^{K} \Exp{\Vert w_k \Vert^2} &\le 2\Exp{\psi(x_0)-\psi(x_{K})} + 4\eta_x \kappa L_{\nabla\Phi}^2\delta_0 + 4\eta_x \kappa K\frac{\sigma^2}{M} + 2 \eta_x K\frac{\sigma^2}{M}\\
      &\quad+\underbrace{\Big(L_{\nabla\phi}+3L_{\nabla\phi}^2\eta_x - \frac{1}{\eta_x} + 2\kappa^4L_{\nabla\Phi}^2\eta_x \Big)}_{=(*)}\sum_{k=0}^{K-1} \Exp{\Vert x_{k+1}-x_k \Vert^2} .
    \end{aligned}
  \end{equation*}
  Applying the step size $\eta_x= \nicefrac{1}{(3 {(1+\kappa)}^2 L_{\nabla\Phi})}$ it follows that
  \begin{equation*}
    (*)\le 2(\kappa+1)L_{\nabla\Phi} - 3{(\kappa+1)}^2L_{\nabla\Phi} + \frac{2\kappa^2L_{\nabla\Phi}}{3} \le -\frac13 {(\kappa+1)}^2L_{\nabla\Phi} \le 0
  \end{equation*}
  which concludes the proof.
\end{proof}

\subsection{Alternating vs simultaneous}%
\label{sub:altvssim-strongly}

Similarly to Section~\ref{sub:altvssim-non-strongly} we want to highlight here the difference in the analysis between the two versions of GDA.\ Again, in this nonconvex-strongly-concave setting the task is to estimate $\delta_k := \Vert y_k^*-y_k \Vert^2$. While in the simultaneous version~\cite{lin2020gradient} obtains for all $k\ge0$ the following inequality
\begin{equation*}
  \begin{aligned}
    \delta_{k+1} &\le (1+\epsilon) \Vert y_{k}^* -y_{k+1} \Vert^2 + (1+\epsilon^{-1})\Vert y_{k+1}^*-y_k^* \Vert^2 \\
        &\le q(1+\epsilon) \delta_{k} + (1+\epsilon^{-1})\Vert y_{k+1}^*-y_k^* \Vert^2,
  \end{aligned}
\end{equation*}
where $q$ is the contraction constant derived from the gradient ascent step and is roughly $1-\frac{1}{\kappa}$. In the alternating version however, we estimate for all $k\ge0$
\begin{equation*}
  \begin{aligned}
    \delta_{k+1} \le q  \Vert y_{k+1}^* -y_{k} \Vert^2 \le q(1+\epsilon) \delta_{k} + q(1+\epsilon^{-1})\Vert y_{k+1}^*-y_k^* \Vert^2.
  \end{aligned}
\end{equation*}
Evidently, in the alternating version the contraction property is applied before the triangle inequality, which leads to the second term being multiplied by $q<1$ as well, influencing the final complexity bound favorably, albeit only slightly.

\section{Numerical Experiments}%
\label{sec:experiments}

In this section, we present several experiments outlining the empirical benefits of alternating GDA over its simultaneous counterpart.

\subsection{Toy example}%

The recent paper~\cite{zhang2022near} showed an improved convergence rate of alternating GDA over the simultaneous version in the strongly convex-strongly concave quadratic setting from $\mathcal{O}(\kappa^{2})$ to $\mathcal{O}(\kappa)$. Inspired by these results we study a nonconvex-strongly-concave toy example
\begin{equation}
  \label{eq:toy-example}
  \min_{x\in \R} \max_{y \in \R}\, -\frac{1}{4}x^2 + xy - \frac{1}{2}y^2,
\end{equation}
 where the resulting max function happens to be a strongly convex quadratic
 \begin{equation}
   \label{eq:toy-max-function}
   \max_{y \in \R}\, \left\{-\frac{1}{4}x^2 + xy - \frac{1}{2} y^2\right\} = \frac{1}{4}x^2.
 \end{equation}
 As we can see from Figure~\ref{fig:toy-performance}, alternating GDA outperforms not only its simultaneous counterpart but also the extra-gradient method (EG)~\cite{extragradient} and the multistep method GDmax~\cite{jin2020local} (employing $10$ ascent steps per descent step). For the Minimax-PPA method~\cite{lin2020near} we only counted iterations but did not account for the computational cost of the double inner procedure which required over 100 solves of a proximally regularized subproblem per iteration. We suspect that for a significantly more ill-conditioned problem the Minimax-PPA methods might have performed more competitively.

\begin{figure}[ht]
  \centering
  \begin{subfigure}[b]{0.45\linewidth}
    \centering
    \includegraphics[width=\linewidth]{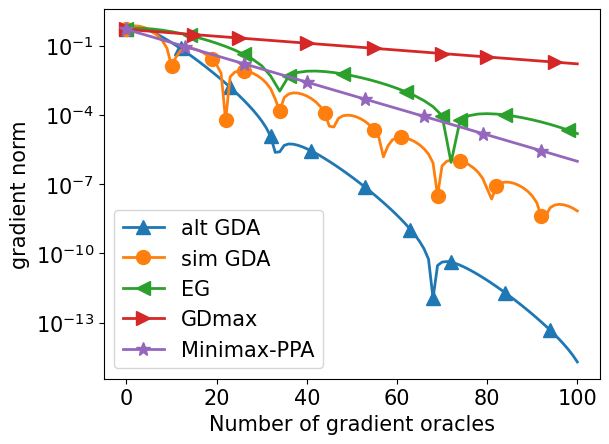}
    \caption{Gradient norm of the \emph{max function}.}
  \end{subfigure}
  \begin{subfigure}[b]{0.45\linewidth}
    \centering
    \includegraphics[width=\linewidth]{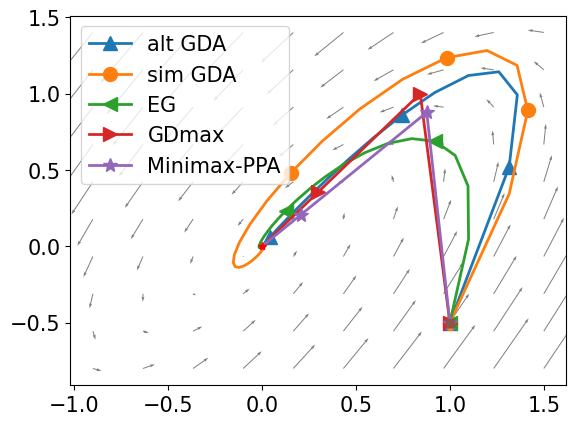}
    \caption{Trajectory of the iterates.}
  \end{subfigure}
  \caption{Comparison of different methods on the $2$-d toy problem~\eqref{eq:toy-example}. The starting point is given by $(x_0, y_0)=(1,-0.5)$, step sizes for the first 4 methods are given by $\eta_x=1/(\kappa L_{\nabla\Phi})$, $\eta_y=1/L_{\nabla\Phi}$ where $\mu=1$ and $L_{\nabla\Phi}\approx 1.78$ and the unique stationary point of the \emph{max function} is $x^*=0$ as evident from~\eqref{eq:toy-max-function}.
EG denotes the extra-gradient method~\cite{extragradient}, GDmax a simple multistep method~\cite{jin2020local} employing $10$ ascent steps for each descent step and Minimax-PPA is the method from~\cite{lin2020near} for which all parameters chosen according to theory.
  }%
  \label{fig:toy-performance}
\end{figure}

In order to account for differences in step sizes we do a grid search across possible step sizes for the two components and plot the number of iteration required to reach a target accuracy, see Figure~\ref{fig:heatmap}. Since the Minimax-PPA method has many inner step sizes and number of inner loop calls to tune it cannot easily be compared here, so we excluded it. We can see that alternating GDA is not only convergent for many different combinations of step sizes but also consistently outperforms the other methods in terms of the required gradient oracle calls.

\begin{figure}[h!]
  \centering
  \begin{subfigure}[b]{0.40\linewidth}
    \centering
    \includegraphics[width=\linewidth]{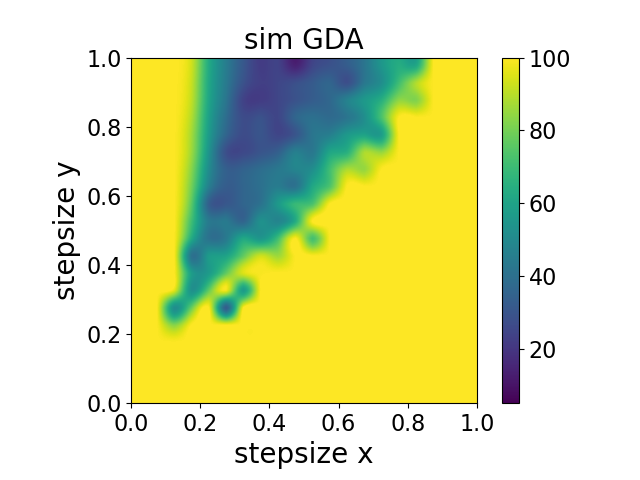}
  \end{subfigure}
  \begin{subfigure}[b]{0.40\linewidth}
    \centering
    \includegraphics[width=\linewidth]{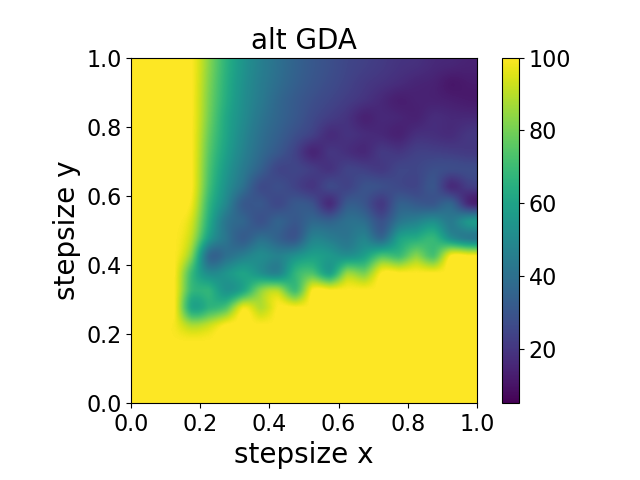}
  \end{subfigure}
  \begin{subfigure}[b]{0.40\linewidth}
    \centering
    \includegraphics[width=\linewidth]{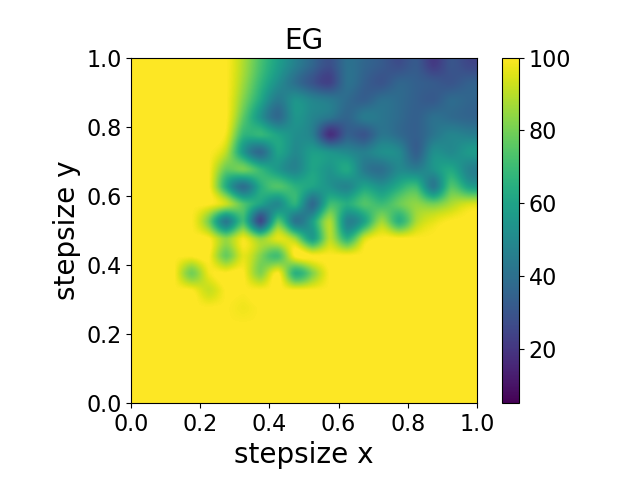}
  \end{subfigure}
  \begin{subfigure}[b]{0.40\linewidth}
    \centering
    \includegraphics[width=\linewidth]{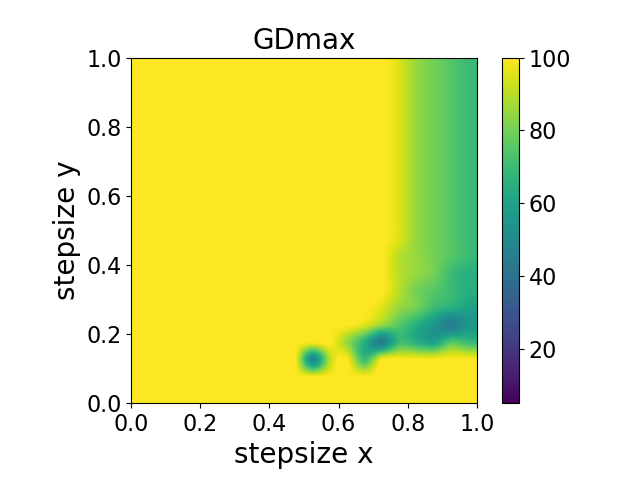}
  \end{subfigure}
  \caption{Gradient oracle calls needed to achieve an iterate gradient norm of the \emph{max function}, see~\eqref{eq:max-function}, smaller than $10^{-4}$ using different methods. }%
  \label{fig:heatmap}
\end{figure}

 We also want to point out that, it might appear from the convergence behavior of the different methods that our toy example~\eqref{eq:toy-example} is easier than the more classical bilinear problem of $\min_x \max_y \, xy$, see~\cite{gidel2019negative,gidel2019variational,numerics-of-gans}, where neither version of GDA converges. While for the latter problem the vector field is always perpendicular to the direction of solution, our new problem exhibits areas where the vector field points away from the solution (see the upper right corner of the Figure~\ref{fig:toy-performance} (b)) and thus exhibits a novel (and challenging behavior) which cannot be captured by bilinear problems.

\subsection{Adversarially robust learning}%

We now highlight the performance of alternating GDA for adversarial learning on the MNIST~\cite{mnist}, Fashion MNIST~\cite{fashion-mnist} and CIFAR10~\cite{cifar10} datasets respectively. We focus on the adversarial learning formulation described in~\eqref{eq:duchi-penalty}, which originated in~\cite{duchi-certifiable} and results in a nonconvex-strongly-concave minimax formulation for large enough $\gamma$.
We use standard convolutional networks (CNN) for all three datasets. For MNIST and Fashion-MNIST we use the architecture proposed by~\cite{duchi-certifiable} of three convolutional layers followed by a dense layer and softmax output. For CIFAR10 we follow the default architecture in the tutorial of~\cite{cleverhans} with seven convolutional layers where the third, fifth and seventh are followed by an average pooling operation.

\begin{figure}[ht]
  \centering
  \begin{subfigure}[b]{0.32\linewidth}
    \centering
    \includegraphics[width=\linewidth]{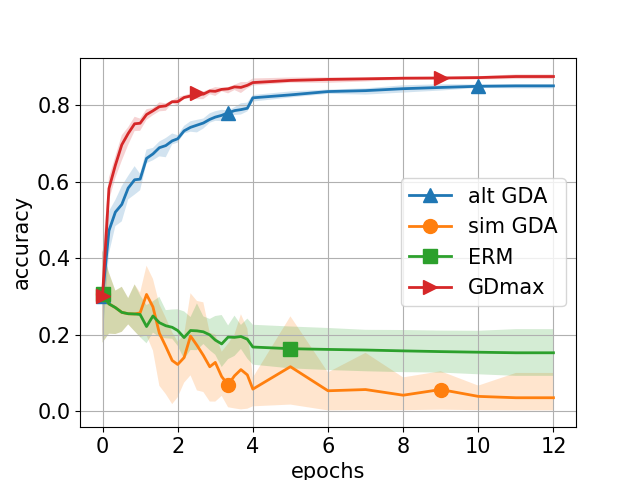}
    \caption{MNIST}
  \end{subfigure}
  \begin{subfigure}[b]{0.33\linewidth}
    \centering
    \includegraphics[width=\linewidth]{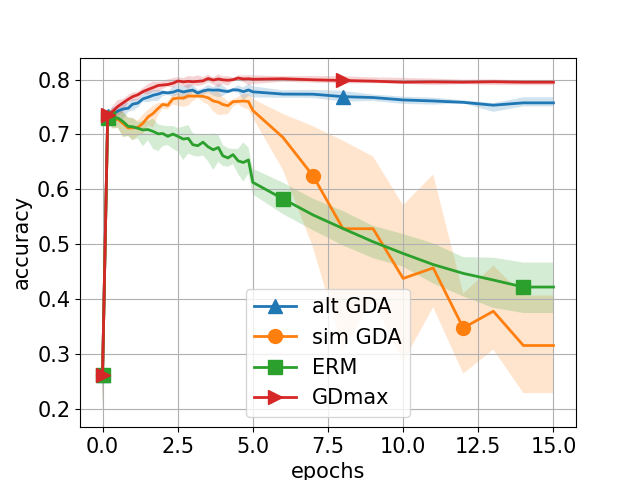}
    \caption{Fashion-MNIST}
  \end{subfigure}
  \begin{subfigure}[b]{0.32\linewidth}
    \centering
    \includegraphics[width=\linewidth]{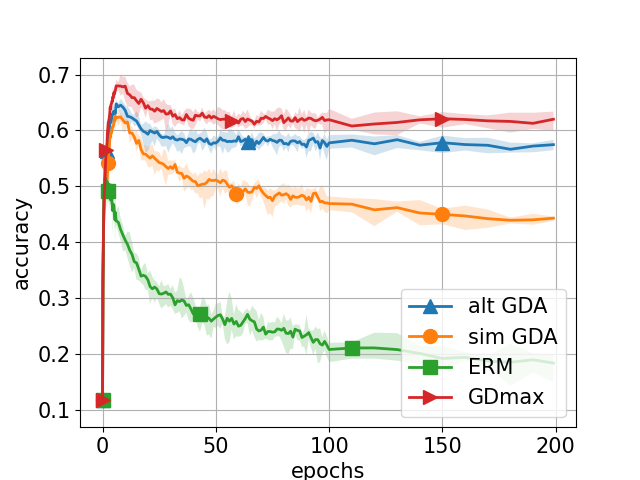}
    \caption{CIFAR10}
  \end{subfigure}
  \caption{Test accuracy (percentage of correctly classified) of different adversarial training methods on multiple datasets averaged over 5 random seeds. ERM denotes empirical risk minimization and represents a standard network trained purely for classification without any consideration of the adversarial examples. For the minimization step we actually used the ADAM optimizer with a learning rate of $0.001$ as done in the code of~\cite{duchi-certifiable,cleverhans}. For the Fashion-MNIST and CIFAR10 datasets we set the Lagrange parameter $\gamma$ to be $1/\gamma=0.4$, whereas for MNIST we use $1/\gamma=1.3$.}%
  \label{fig:accuracy}
\end{figure}

Since the robust training did not significantly impact the performance of the CNNs on the clean test examples on any of the data sets, we only report the performance on the adversarial examples.

In contrast to multistep methods, as proposed in~\cite{duchi-certifiable,jin2020local,nouiehed2019solving} which aim to (approximately) solve the inner maximization problem and typically start in every iteration from either the ``clean'' training example (or a randomly perturbed point~\cite{random-perturbed-start}), GDA can be interpreted as a warm starting procedure which stores the adversarial example computed in the previous epoch. This also emphasizes the advantage of alternating GDA (and why it can be considered the more natural approach) as this method uses computed adversarial examples right away whereas the simultaneous version stores them to only use them in the next epoch.

Although this is not the main focus this work, we also contrast the single step methods with a method approximately solving the maximization problem (GDmax, see~\cite{lin2020gradient}) and observe that while the multistep method slightly outperforms alternating GDA, see Figure~\ref{fig:accuracy}, this comes at the cost of an $\approx8$ times higher computation time (based on $\approx15$ gradient ascent steps for the multistep method).

Nevertheless, all results clearly show that alternating GDA consistently outperforms simultaneous GDA without any additional computational cost.

\section{Conclusion}%

We show novel complexity results for the alternating gradient descent ascent method for nonconvex-(strongly) concave minimax problems using stochastic or deterministic gradient evaluations. Since these bounds are only a first step into the analysis of alternating GDA in this sophisticated setting, they do not explain theoretically the benefit of the alternating version. 
However, we provide empirical evidence that this method is favorable.


\end{document}